\documentclass[10pt]{amsart}

\usepackage[a4paper, total={6in,8in}]{geometry}
\usepackage[utf8]{inputenc}
\usepackage[T1]{fontenc}
\usepackage[english]{babel}
\usepackage{amsmath}
\usepackage{amsfonts}
\usepackage{amssymb}
\usepackage{amsthm}
\usepackage{bbm}
\usepackage{csquotes}
\usepackage{lipsum}
\usepackage{esint}
\usepackage{mathtools}
\usepackage{tikz}
\usepackage{enumitem}
\usepackage{graphicx}
\usetikzlibrary{arrows}

\newcommand{\abar}{\overline{a}}
\newcommand{\bbar}{\overline{b}}
\newcommand{\cbar}{\overline{c}}
\newcommand{\Abar}{\overline{A}}
\newcommand{\Bbar}{\overline{B}}
\newcommand{\Cbar}{\overline{C}}

\newcommand{\pdv}[2]{\frac{\partial #1}{\partial #2}}
\newcommand{\mpdv}[3]{\frac{\partial^2 #1}{\partial #2 \partial #3}}
\newcommand{\intxv}{\int_{xv}}
\newcommand{\inttxv}{\int_{txv}}

\newcommand{\weakstarto}{\overset{\ast}{\rightharpoonup}}
\newcommand{\weakto}{\rightharpoonup}

\newcommand{\Rxv}{\mathbb{R}^{2N}_{x,v}}

\newcommand{\vertiii}[1]{{\left\vert\kern-0.25ex\left\vert\kern-0.25ex\left\vert #1
		\right\vert\kern-0.25ex\right\vert\kern-0.25ex\right\vert}}

\makeatletter
\newcommand{\proofpart}[2]{%
	\vspace{8pt}
	\par
	\addvspace{\medskipamount}%
	\noindent\emph{Step #1: #2}\par\nobreak
	\addvspace{\smallskipamount}%
	\@afterheading
}
\makeatother

\DeclareMathOperator{\supp}{supp}
\DeclareMathOperator{\Div}{div}

\newtheorem{definition}{Definition}
\newtheorem{lemma}{Lemma}
\newtheorem{theorem}{Theorem}
\newtheorem{proposition}{Proposition}
\newtheorem{remark}{Remark}

\begin{document}
	
\title{On the semi-classical limit for the Landau-Fermi-Dirac equation}
\author{Paulo Sampaio}
\address{Centre de Mathématiques Laurent Schwartz (CMLS), CNRS, École polytechnique, Institut Polytechnique de Paris, 91120 Palaiseau, France}
\email{paulo.alves-sampaio@dauphine.psl.eu}


\begin{abstract}
	We study sequences of solutions to the inhomogeneous Landau-Fermi-Dirac equation with Coulomb potential in which the quantum parameter converges to zero.
	Our main result establishes the compactness of these sequences, which allows us to show that, up to a subsequence, these solutions converge to a renormalized solution of the classical Landau equation with a defect measure, as defined by Villani.
	
	To do this, we work in the class of solutions that are obtained through approximation procedures.
	For these solutions, we were able to show compactness in the vanishing quantum parameter limit through a diagonal argument, which combines techniques from the study of Cauchy problems for both the classical Landau and the Landau-Fermi-Dirac equations.
\end{abstract}

\maketitle

\section{Introduction}

In 1936, Landau \cite{landau-translated-1936} proposed, based on phenomenological arguments, an equation that models the behavior of dilute plasmas, today known as the Landau equation,
\begin{equation}
	\label{ch2:eq:landau}
	\pdv{f}{t} + v \cdot \nabla_x f
	= \Div_v
	\left(
		\int a(v - v_*) \big(
					f_* \nabla_v f -
					f  \nabla_{v_*}f_*
					\big) \; dv_*
	\right).\\
\end{equation}
The unknown $f = f(t,x,v)$ represents the probability density of finding a particle at time $t$, at position $x$, with velocity $v$.
Above, we have used the abbreviation $f_*$ for the function $f(t, x, v_*)$ as well as the convention $(\Div_v M)_i = \sum_{j=1}^{N} \pdv{M_{ij}}{v_j}$ for the divergence of a matrix function $M$.

The left-hand side is the transport part of the equation, which models the inertial aspect of the particles' movement.
In fact, if the right-hand side of the equation were zero, then \eqref{ch2:eq:landau} would reduce to the free transport equation, and the particles would then follow straight trajectories along the characteristic lines.

The right-hand side, on the other hand, models how these trajectories are affected due to binary collisions between the particles.
The matrix $a(v-v_*)$ inside the integral is called the \emph{collision kernel} and typically has the form
\begin{equation}
	\label{ch2:collision-kernel-formula}
	a(z) = \Gamma(|z|) \left(I_N - \frac{z \otimes z}{|z|^2}\right),
\end{equation}
where $I_N$ denotes the $N \times N$ identity matrix and $\Gamma(|z|)$ is a radial function, called the \emph{cross section}, which varies according to the interaction potential between the particles.
This way, if the particles interact through a potential of the type
\[
	U(x) = \frac{k}{|x|^s},
\]
then it can be shown that, for $N=3$, $\Gamma(|z|) = K |z|^\theta$, where $\theta = 3-4/s$, for $s \geq 2$ (see, for example, \cite{desvillettes_1992}).
The cases typically considered in the literature correspond to when $\Gamma(|z|) = K |z|^{\gamma+2}$, where $-N \leq \gamma \leq 1$.
If $0 < \gamma \leq 1$, we say the particles interact through a \emph{hard potential}, and if $\gamma < 0$, we say we are dealing with a \emph{soft potential}.

Although not the way it was originally deduced, equation \eqref{ch2:eq:landau} is generally seen as the asymptotic limit of the Boltzmann equation when the collisions are predominantly grazing, i.e. the angle of deviation in each collision is small.
This type of behavior is expected of very hot and rarefied gases and, therefore, the Landau equation finds an important application in plasma physics.

In this case, the particles interact with each other through a \emph{Coulomb} potential, and we have $\gamma = -N$.
Among the physical potentials discussed in the literature, this is the only one that has a clear physical significance, but it is also the one that exhibits the most pronounced singularity at $z=0$.
For this reason, we will consistently keep this potential in mind when formulating hypotheses and stating results throughout this paper, so as to ensure that all results remain valid, at least in this case.
In fact, we will see later that this singularity poses problems even in defining a solution to the equation, and we will see that the solutions constructed in the literature for the equation with this potential are of the renormalized type.

The above description only assumes that the particles in question are point-like, as in classical mechanics, and obey Maxwell-Boltzmann type statistics.
If, on the other hand, we want a model for quantum particles, some changes must be taken into account.

Take the case of fermions.
These are particles that obey the Pauli exclusion principle, which states that no two particles can be in the same quantum state at the same time.
Statistically, this implies that the distribution function must have at most one particle per $h^3/m^3$ of space, implying the pointwise bound
\begin{equation}
	\label{ch2:ineq:pauli-exclusion}
	0 \leq f(t,x,v) \leq \frac{1}{\varepsilon}
	\equiv \frac{h^3}{m^3 \beta}.
\end{equation}

Above, we have that $h$ is the Planck constant, $m$ is the mass of the particle, and $\beta$ is called the degeneracy of the particle.
Fermions are particles that undergo a process of \emph{saturation}: if we already have too many particles in a given region of phase space, it becomes less likely that even more particles will enter it.
Thus, the collision probabilities for the quantum case are different from those of the classical case, and this new underlying statistics is called Fermi-Dirac.
Adapting the Landau equation to these new probabilities gives rise to the so-called Landau-Fermi-Dirac (LFD) equation, which reads
\begin{equation}
	\label{ch2:eq:lfd}
	\pdv{f}{t} + v \cdot \nabla_x f
	= \Div_v
	\left(
		\int a(v - v_*) \big(
					f_* (1 - \varepsilon f_*) \nabla_v f -
					f (1 - \varepsilon f) \nabla_{v_*}f_*
					\big) dv_*
	\right). \\
\end{equation}

Historically, the first adaptation of the kinetic theory of gases to a quantum version began in 1928 with Nordheim \cite{nordheim-1928}, followed shortly after by Uehling and Uhlenbeck \cite{uehling-1933}, who adapted the Boltzmann equation to take Fermi-Dirac statistics into account, giving rise to the Boltzmann-Fermi-Dirac equation.

In 1979, Danielewicz \cite{danielewicz-1980} used the asymptotic methods used to derive the Landau equation from Boltzmann's to take the limit of the Boltzmann-Fermi-Dirac equation, effectively arriving at the Landau-Fermi-Dirac equation.
The same equation, however, had already appeared in the literature in 1970, in the works of Kadomtsev and Pogutse on the relaxation of stellar systems \cite{kadomtsev_1970}.

Once in possession of a quantum version of a classical equation, one wonders about the compatibility between the two.
In fact, the correspondence principle of quantum mechanics states that the behavior of quantum systems should be close to that predicted by classical mechanics in the limit of a small Planck constant.
Indeed, if we formally take the limit $\varepsilon \to 0$ in equation \eqref{ch2:eq:lfd}, we recover equation \eqref{ch2:eq:landau}.

In this article, we propose to show that the above procedure is rigorously justified.
That is, solutions of the Landau-Fermi-Dirac equation converge to solutions of the Landau equation in the limit of a small quantum parameter $\varepsilon$, a procedure mathematically called a \emph{semi-classical limit}.
More specifically, we will prove that the LFD solutions constructed in \cite{sampaio-2024} do in fact converge to solutions of the Landau equation constructed by Villani in \cite{villani_1996}, and therefore the Cauchy theories constructed for these two equations are compatible.
Furthermore, such a result also justifies the use of the LFD equation as an approximation of the Landau equation.

A result of this type was proven by Dolbeault \cite{dolbeault_1994} in his pioneering work on the Cauchy problem for the Boltzmann-Fermi-Dirac (BFD) equation.
In this work, in addition to constructing weak solutions for this equation, it is shown that these solutions satisfy a semi-classical limit: From a sequence of solutions $f_n$ of BFD with vanishing quantum parameter, we can extract a subsequence converging to a renormalized solution of the Boltzmann equation, as defined by DiPerna and Lions in \cite{diperna-lions-1989}.

More recently, He, Lu, and Pulvirenti \cite{he-2021} established a semi-classical limit for the \emph{homogeneous} Boltzmann-Fermi-Dirac (BFD) equation, i.e., the case where the distribution function $f$ depends only on $(t,v)$.
Motivated by physical considerations, they studied collision kernels that scale with the quantum parameter.
In this setting, they showed that a sequence of solutions to the BFD equation with vanishing quantum parameter converges, up to extraction of a subsequence, to a solution of the homogeneous classical Landau equation.
To the best of the author’s knowledge, no semi-classical limit results are currently available for the LFD equation.

\section{Preliminaries and main result}
Before stating the central result, let us briefly review the Cauchy theory for the Landau and LFD equations.
For the Landau equation, one can show that formally
\begin{equation}
	\label{ch2:eq:conservation-laws-formal}
	\frac{d}{dt} \iint_{\Rxv} f(t) \varphi \, dxdv = 0
\end{equation}
if $\varphi = 1, v_i, |v|^2$ or $|x - vt|^2$, which physically corresponds to the conservation of mass, linear momentum, kinetic energy, and moment of inertia, respectively.
On the other hand, if we define the entropy
\begin{equation}
	\label{ch2:eq:classical-entropy}
	H(t) = \iint_{\Rxv} f \log f \, dxdv,
\end{equation}
we have that, formally, if $f$ is a solution of \eqref{ch2:eq:landau}, then the functional $H(t)$ is decreasing.

Thus, it is reasonable to assume that these quantities are bounded at the time $t=0$, i.e.,
\begin{equation}
	\label{ch2:ineq:f_0-bounds}
	\iint_{\Rxv} f_0 (1 + |x|^2 + |v|^2 + \log f_0)  \, dxdv < \infty
\end{equation}
for a given initial $f_0$.
Ideally, we would like to construct a solution that obeys the conservation laws and entropy decay.
In particular, if we had such a solution, then the quantity
\[
	\sup_{t > 0} \iint_{\Rxv} f(t) (1 + |x|^2 + |v|^2 + \log f(t))  \, dxdv
\]
would be bounded by \eqref{ch2:ineq:f_0-bounds}, and so we look for our solution in the space of functions such that this quantity is finite.

Here, the singularity of the collision kernel comes into play.
Indeed, if we consider this kernel for a Coulomb-type potential, the above bounds do not guarantee that the collision integral on the right-hand side of \eqref{ch2:eq:landau} is well-defined, even in the sense of the distributions.
One way to get a sense of the solution even with this limitation is to consider \emph{renormalized solutions}.

Formally, the renormalization process is a change in the unknown of the equation.
Instead of solving the Landau equation for $f$, we will choose a smooth, bounded function $\beta$, and try to solve the equation for $g = \beta(f)$.
Since the $\beta$ function is bounded, ideally we get an unknown $g$ that we know is a bounded function a priori, which can help smooth out the product of singularities.

So, let $\beta = \beta(t)$ be a smooth function, called a non-linearity.
Multiplying the Landau equation by $\beta'(f)$, we get, using the repeated index summation convention,
\[
	\pdv{\beta(f)}{t} + v \cdot \nabla_x \beta(f)
	= \beta'(f) \pdv{}{v_i} \left\{
		\abar_{ij} \pdv{f}{v_j} - \bbar_i f
	\right\},
\]
where $\abar_{ij} = a_{ij} *_v f$ and $\bbar_i = \pdv{\abar_{ij}}{v_j}$.
We can rewrite this equation as
\begin{equation}
	\label{ch2:landau-renorm}
	\begin{split}
	\pdv{\beta(f)}{t} + v \cdot \nabla_x \beta(f)
	= &\pdv{}{v_i}\left\{
		\abar_{ij} \pdv{\beta(f)}{v_j} - \bbar_i \beta(f)
	\right\}\\
	&  - \beta''(f) \abar_{ij} \pdv{f}{v_i} \pdv{f}{v_j}
	+ \cbar [\beta(f) - \beta'(f) f],
	\end{split}
\end{equation}
where $\cbar = \mpdv{\abar_{ij}}{v_i}{v_j}$, and we say that $f$ is a \emph{renormalized solution} of the Landau equation if the distribution $\beta(f)$ solves \eqref{ch2:landau-renorm} in the sense of distributions.

Let $N \geq 2$.
We assume throughout the paper that in the Landau equation we have $a(z)$ of the form \eqref{ch2:collision-kernel-formula}
and that it satisfies
\begin{equation}
    \label{ch2:property:a-regularity}
    \begin{split}
	    a_{ij}, \quad \sum_j \pdv{a_{ij}}{z_j} &\in L^1(\mathbb{R}^N) + L^\infty(\mathbb{R}^N), \\
	    \sum_{i,j} \mpdv{a_{ij}}{z_i}{z_j} &\in \mathcal{M}(\mathbb{R}^N) + L^\infty(\mathbb{R}^N),
    \end{split}
\end{equation}
where $\mathcal{M}(\mathbb{R}^N)$ is the space of bounded nonnegative measures in $\mathbb{R}^N$.
We also suppose that the collision kernel satisfies the integrability condition
\begin{equation}
\label{ch2:property:collision-kernel-integrability}
	\{ z \mapsto \Gamma(|z|) \} \in L^r_{loc}(\mathbb{R}^N),
\end{equation}
for some $r > 1$ if $N = 2$ and $r = 1$ if $N \geq 3$.
Notably, the three-dimensional Coulomb case satisfies all of these assumptions.

The typical strategy for constructing solutions to this type of equation is to take solutions to a well-chosen approximate equation, whose Cauchy problem is easier to study, and then prove a compactness theorem for the solutions, enabling us to pass this equation to the limit.

For the Landau equation, such a compactness theorem was obtained by Lions in \cite{lions_1994}.
However, as shown by Villani in \cite{villani_1996}, the passage to the limit of the quadratic term in the derivatives $\pdv{f}{v_i}$ poses problems and a defect measure appears, giving rise to an even weaker notion of solution.

Let us then denote by $L^1_2(\Rxv)$ as the space of functions $g$ such that
\[
	\iint_{\Rxv} |g(x,v)| (1 + |x|^2 + |v|^2) \; dxdv < \infty.
\]
Unless otherwise specified, all integrals with respect to $x$ and $v$ are understood to be over the entire phase space $\Rxv$.
The existence result from \cite{villani_1996} then reads as follows.

\begin{definition}
	\label{ch2:def-renormalized-landau}
	Let $f_0 \in L^1_2(\Rxv)$ be such that $f_0 \geq 0$, and for every $\delta > 0$, let $\beta_\delta(t) = \frac{t}{1+\delta t}$.
	We say that $f \in \mathcal{C}((0,\infty); \mathcal{D}'(\Rxv)) \cap L^\infty((0,\infty); L^1_2(\Rxv))$ is a renormalized solution of the Landau equation with defect measure and initial data $f_0$ if, for every $\delta > 0$, there exists a nonnegative measure $\mu_\delta$, bounded on all sets $(0,T) \times \mathbb{R}^N_x \times B_R$, $T < \infty$, such that $\beta_\delta(f)$ satisfies
	\begin{align}
		\label{ch2:eq:landau-renormalized}
		\pdv{\beta_\delta(f)}{t} + v \cdot \nabla_x \beta_\delta(f)
		= &\pdv{}{v_i} \left\{
			\abar_{ij} \pdv{\beta_\delta(f)}{v_j} - \bbar_i \beta_\delta(f)
			\right\}\nonumber\\
			&- \beta_\delta''(f) \abar_{ij} \pdv{f}{v_i} \pdv{f}{v_j}
			+ \cbar [\beta_\delta(f) - \beta_\delta'(f) f]
			+ \mu_\delta
	\end{align}
	in $\mathcal{D}'((0,\infty) \times \Rxv)$, $f(t) \to f_0$ as $t \to 0^+$ in $\mathcal{D}'(\Rxv)$, and we have

	1) Conservation of mass and linear momentum:
	\begin{align*}
		&\iint f(t,x,v) \, dxdv
		= \iint f_0(x,v) \,dxdv,
	\\&
		\iint f(t,x,v) v_i \, dxdv
		= \iint f_0(x,v) v_i \,dxdv
		\hspace{10pt} \forall i \in \{ 1, \cdots, N \}.
	\end{align*}

	2) Decay of kinetic energy and moment of inertia:
	\begin{align*}
		&\iint f(t,x,v) |v|^2 \,dxdv
		\leq \iint f_0(x,v) |v|^2 \,dxdv,
	\\&
		\iint f(t,x,v) |x-tv|^2 \,dxdv
		\leq \iint f_0(x,v) |x-tv|^2 \,dxdv.
	\end{align*}

	3) Entropy inequality:
	\begin{equation*}
		\begin{split}
		&\iint f(t,x,v) \log f(t,x,v)\,dxdv+ \frac{1}{2} \int_{0}^{t}\iint d(\tau,x,v)\,dxdvd\tau\\
		& \leq \iint f_0(x,v) \log f_0(x,v)\,dxdv,
		\end{split}
	\end{equation*}
	where
	\begin{equation}
		\label{ch2:eq:entropy-dissip-classical}
		d(t,x,v) =
		\int a_{ij}(v-v_{*})ff_{*}
		\left(\frac{\partial_{v_i} f}{f}-\frac{\partial_{v_{*,i}}f_{*}}{f_{*}}\right)
		\left(\frac{\partial_{v_j} f}{f}-\frac{\partial_{v_{*,j}}f_{*}}{f_{*}}\right)
		\, dv_{*}
	\end{equation}
	is the entropy dissipation.
\end{definition}

More precisely, a renormalized solution with defect measure of the Landau equation satisfies, for every $\varphi \in C^\infty_c([0,\infty) \times \Rxv)$,
\begin{equation*}
\begin{split}
    &-\int_0^\infty \iint \beta_\delta(f)\left( \pdv{\varphi}{t} + v \cdot \nabla_x \varphi \right) \, dtdxdv \\
    &= \int_0^\infty \iint \abar_{ij} \beta_\delta(f) \mpdv{\varphi}{v_i}{v_j} \, dtdxdv \\
    &\quad + \int_0^\infty \iint \left[
    	\pdv{\abar_{ij}}{v_j} \beta_\delta(f)
    	+ \bbar_i \beta_\delta(f)
    	\right] \pdv{\varphi}{v_i} \, dtdxdv \\
    &\quad + \int_0^\infty \iint \left( - \beta_\delta''(f) \abar_{ij} \pdv{f}{v_i} \pdv{f}{v_j} + \cbar [ \beta_\delta(f) - f \beta_\delta'(f)] \right) \varphi \, dtdxdv \\
    &\quad + \int \varphi \, d\mu_\delta.
\end{split}
\end{equation*}
We also have, for every $\psi \in \mathcal{D}(\mathbb{R}^N_x \times \mathbb{R}^N_v)$,
\[
	\iint f(t) \psi \, dxdv \; \xrightarrow{t \to 0^+} \iint f_0 \psi \, dxdv.
\]

Some nuances of the above definition are worth highlighting.
Note that while the equation on $(0,\infty) \times \Rxv$ is stated for $\beta_\delta(f)$, the statement of the initial condition rests solely on the function $f$, without renormalization.
In addition, some terms in the above statement are not well-defined or do not have a canonical definition.

\begin{remark}
	\label{ch2:rem:product-rule-interp}
	In equation \eqref{ch2:eq:landau-renormalized}, since we have no regularity for the derivatives of $f$, the term $\abar_{ij} \pdv{\beta_\delta(f)}{v_j}$ must be interpreted as $\pdv{}{v_j} \left[\abar_{ij} \beta_\delta(f) \right] - \pdv{\abar_{ij}}{v_j} \beta_\delta(f)$, in the sense of distributions.
\end{remark}

\begin{remark}
	\label{ch2:rem:quadratic-term}
	Since we do not suppose any regularity upon renormalized solutions, one might ask what the definition of the quadratic term $- \beta''(f) \abar_{ij} \pdv{f}{v_i} \pdv{f}{v_j}$ in the sense of distributions is.
	One such interpretation is to view this as a notation for
	\begin{equation}
		\label{ch2:eq:formula-for-quad-term}
		\int_{\mathbb{R}^N}
		\left|
			\sqrt{a(v-v_*)} \sqrt{f_*} \nabla_v \gamma_\delta(f)
		\right|^2 \, dv_*,
	\end{equation}
	where $\gamma_\delta$ is defined such that $(\gamma_\delta'(z))^2 = -\beta_\delta''(z)$.
	The vector field inside the square is an $L^2_{loc}((0,\infty) \times \mathbb{R}^{2N}_{x,v}; L^2(\mathbb{R}^N_{v_*}))$ function, defined in the sense of distributions as
	\begin{equation*}
	\begin{split}
	    \sqrt{a(v-v_*)} \sqrt{f_*} \nabla_v \gamma_\delta(f)
	    &= \Div_{v} \Big[ \sqrt{a(v-v_*)} \sqrt{f_*} \gamma_\delta(f) \Big] \\
	    &\quad - \Div_{v} \Big[ \sqrt{a(v-v_*)} \Big] \sqrt{f_*} \gamma_\delta(f).
	\end{split}
	\end{equation*}

	This way, the term $- \beta''(f) \abar_{ij} \pdv{f}{v_i} \pdv{f}{v_j}$ is an $L^1_{loc}((0,\infty) \times \Rxv)$ function, independent of the sequence of approximating functions $(f^n)_n$.
	It should be noted, however, that since this is a formal representation, the vector field above is not necessarily a pointwise product of functions.
\end{remark}

\begin{remark}
	\label{ch2:rem:entropy-dissip-classical}
	In the same vein as the definition of the quadratic term above, the entropy dissipation \eqref{ch2:eq:entropy-dissip-classical} should be interpreted as the $L^2((0,\infty) \times \Rxv)$ norm of the function
	\begin{align*}
		&2 \Div_{v} \left( \sqrt{a(v-v_*)} \sqrt{f} \sqrt{f_*} \right)
		- 2 \Div_v \left( \sqrt{a(v-v_*)} \right) \sqrt{f} \sqrt{f_*}\\
		&- 2 \Div_{v_*}\left( \sqrt{a(v-v_*)} \sqrt{f} \sqrt{f_*} \right)
		+ 2 \Div_{v_*}\left( \sqrt{a(v-v_*)} \right) \sqrt{f} \sqrt{f_*}.
	\end{align*}
\end{remark}

For the Landau-Fermi-Dirac (LFD) equation, we have the same conservation laws \eqref{ch2:eq:conservation-laws-formal} of mass, momentum, kinetic energy, and moment of inertia.
However, the saturation effect described by the Pauli exclusion principle leads us to consider, instead of \eqref{ch2:eq:classical-entropy}, the \emph{quantum} entropy
\begin{equation}
	\label{ch2:eq:quantum-entropy}
	S_\varepsilon(g) = -\frac{1}{\varepsilon} \iint \varepsilon g \log (\varepsilon g) + (1-\varepsilon g) \log(1-\varepsilon g) \, dxdv,
\end{equation}
which is a decreasing functional in the solutions of the Landau-Fermi-Dirac equation \eqref{ch2:eq:lfd}.

For the quantum case, the extra bound \eqref{ch2:ineq:pauli-exclusion} implies that the collision kernel is well-defined even in the case of a Coulomb-type potential.
Thus, we do not need renormalization to show the existence of a global solution, and we can construct a weak solution, in the sense of distributions, without any problems.

This way, the existence result from \cite{sampaio-2024} reads

\begin{definition}
	\label{ch2:def-weak-solution}
	Let $f_0 \in L^1_2(\Rxv)$ be such that $0 \leq f_0 \leq \varepsilon^{-1}$.
	A function $f = f(t,x,v)$ in $C((0,\infty); \mathcal{D}'(\Rxv)) \cap L^\infty((0,\infty); L^1_2(\Rxv))$ is called a weak solution of LFD with initial data $f_0$ and quantum parameter $\varepsilon > 0$ if it satisfies
	\[
		\pdv{f}{t} + v \cdot \nabla_x f
		= \mpdv{( \abar_{ij} f )}{v_i}{v_j}
		+ \pdv{}{v_i} \left[
			\pdv{\abar_{ij}}{v_j} f + \bbar_i f(1 - \varepsilon f)
			\right]
	\]
	in $\mathcal{D}'((0,\infty) \times \Rxv)$, with $ \abar_{ij} = a_{ij} \ast_v (f(1- \varepsilon f)) $, $\bbar_i = \pdv{a_{ij}}{v_j} \ast_v f$, and $f(t) \to f_0$ in $\mathcal{D}'(\Rxv)$, and, moreover, we have

	1) Pauli exclusion principle:
	\[
		0 \leq f(t) \leq \varepsilon^{-1}.
	\]

	2) Conservation of mass and linear momentum:
	\begin{align*}
		&\iint f(t,x,v) \, dxdv
		= \iint f_0(x,v) \,dxdv,
	\\&
		\iint f(t,x,v) v_i \, dxdv
		= \iint f_0(x,v) v_i \,dxdv
		\quad \forall i \in \{ 1, \cdots, N \}.
	\end{align*}

	3) Decay of kinetic energy and moment of inertia:
	\begin{align*}
		&\iint f(t,x,v) |v|^2 \,dxdv
		\leq \iint f_0(x,v) |v|^2 \,dxdv,
	\\&
		\iint f(t,x,v) |x-tv|^2 \,dxdv
		\leq \iint f_0(x,v) |x-tv|^2 \,dxdv.
	\end{align*}

	4) Entropy inequality:
	\[
		\iint s(t,x,v)\,dxdv + \frac{1}{2} \int_{0}^{t}\iint d(\tau,x,v)\,dxdvd\tau \leq \iint s(0,x,v)\,dxdv,
	\]
	where
	\[
		s(t,x,v) = \varepsilon f\log (\varepsilon f) + (1- \varepsilon f)\log(1- \varepsilon f)
	\]
	is the quantum entropy and
	\begin{equation}
		\label{ch2:dissipation-def}
		\begin{split}
			d(t,x,v) =&
			\int a_{ij}(v-v_*) f (1-\varepsilon f) f_* (1 - \varepsilon f_*)
\left(
				\frac{\partial_{v_i} f}{f(1-\varepsilon f)} - \frac{\partial_{v_{*,i}}f_*}{f_*(1-\varepsilon f_*)}
			\right) \\
			&\times \left(
				\frac{\partial_{v_j} f}{f(1-\varepsilon f)} - \frac{\partial_{v_{*,j}}f_*}{f_*(1-\varepsilon f_*)}
			\right) \,dv_*
		\end{split}
	\end{equation}
	is the quantum entropy dissipation.
\end{definition}

More precisely, $f$ is a global weak solution of LFD if for every test function $\varphi \in \mathcal{D}((0, \infty) \times \mathbb{R}^N_x \times \mathbb{R}^N_v))$, we have
\begin{equation*}
\begin{split}
    &\int_0^\infty \iint f \pdv{\varphi}{t} \, dxdvdt
+ \int_0^\infty \iint f v_i \pdv{\varphi}{x_i} \, dxdvdt \\
    &= - \int_0^\infty \iint \abar_{ij} f \mpdv{\varphi}{v_i}{v_j} \, dxdvdt \\
    &\quad - \int_0^\infty \iint \left[
        \pdv{\abar_{ij}}{v_j} f + \bbar_i f(1 - \varepsilon f)
        \right] \pdv{\varphi}{v_i} \, dxdvdt,
\end{split}
\end{equation*}
and for every $\psi \in \mathcal{D}(\mathbb{R}^N_x \times \mathbb{R}^N_v)$, we have
\[
	\iint f(t) \psi \, dxdv \; \xrightarrow{t \to 0^+} \iint f_0 \psi \, dxdv.
\]

\begin{remark}
	\label{ch2:rem:entropy-dissipation-definition}
	As is the case for the classical Landau equation, here the lack of regularity a priori in $v$ implies that one should interpret the entropy dissipation \eqref{ch2:dissipation-def} in a different way.
	The result in \cite{sampaio-2024} proves that the entropy inequality is valid if we interpret the integral $\int_{0}^{t}\iint d(\tau,x,v)\,dxdvd\tau$ as a notation for the $L^2((0,t) \times \mathbb{R}^N_x \times \mathbb{R}^N_v \times \mathbb{R}^N_{v_*})$ norm of
	\begin{equation*}
	\begin{split}
	    &2 \Div_v \left[ \sqrt{a(v-v_*)} \sqrt{f_* (1 -\varepsilon f_*)} \frac{1}{\sqrt{\varepsilon}} \arcsin \sqrt{\varepsilon f} \right] \\
	    &- 2 \Div_{v_*} \left[ \sqrt{a(v-v_*)} \sqrt{f (1 - \varepsilon f)} \frac{1}{\sqrt{\varepsilon}} \arcsin \sqrt{\varepsilon f_*} \right] \\
	    &- 2 \Div_v \left( \sqrt{a(v-v_*)} \right) \sqrt{f_* (1 - \varepsilon f_*)} \frac{1}{\sqrt{\varepsilon}} \arcsin \sqrt{\varepsilon f} \\
	    &+ 2 \Div_{v_*} \left(\sqrt{a(v-v_*)}\right) \sqrt{f (1 - \varepsilon f)} \frac{1}{\sqrt{\varepsilon}} \arcsin \sqrt{\varepsilon f_*},
	\end{split}
	\end{equation*}
	where the divergence of a matrix $A$ is defined as $\Div_v A = \sum_{j} \pdv{A_{ij}}{v_j}$.
	The motivation for using the expression \eqref{ch2:dissipation-def} is that if $f$ has some regularity in $v$, then the two expressions are equal.
\end{remark}

The existence of such solutions was proven in \cite{sampaio-2024} by a compactness argument reminiscent of the work of Lions and Villani for the Landau equation.
In this approach, we construct a sequence of equations that adequately approximates the LFD equation, and then deduce a compactness theorem for its solutions, which allows us to pass to the limit in the approximations and obtain a solution to the LFD equation.

The technique we will use to prove a semi-classical limit for solutions of LFD depends on being able to pass from a weak formulation of this equation to a renormalized one, and therefore to manipulate products of derivatives of these solutions.
We do not have regularity results for these solutions, but it turns out that it is enough to consider solutions that originate from solutions to approximate, more regular equations, which leads us to the definition of \emph{suitable} weak solutions of the LFD equation.

\begin{definition}
	\label{ch2:def:suitable-weak-solution}
	A weak solution $f$ to the LFD is called a \emph{suitable weak solution} if there exists a sequence $f_m \in L^2((0,T) \times \mathbb{R}^N_x; H^1(\mathbb{R}^N_v))$ of solutions to
	\begin{equation}
		\label{ch2:regularized}
		\begin{cases}
			\displaystyle
			\pdv{f_m}{t} + v \cdot \nabla_x f_m =
			\pdv{}{v_j} \left[
				\abar^m_{ij} \pdv{f_m}{v_i} - \bbar^m_j f_m (1 - \varepsilon f_m)
				\right]\\
			f_m|_{t=0} = f_{0,m},
		\end{cases}
	\end{equation}
	such that $f_m \to f$ in $L^1_{loc}((0,\infty); L^1(\Rxv))$ and

	\begin{enumerate}[label=\Roman*.]
	\item \textbf{Convergence of the coefficients:}
	$\abar^m, \Div_v \abar^m, \bbar^m, \Div_v \bbar^m$ are $L^\infty_{loc}((0,T) \times \Rxv)$ functions such that
	\begin{alignat*}{2}
		\abar^m &\to \abar \hspace{1cm}
		&\bbar^m &\to \bbar\\
		\Div_v \abar^m &\to \Div_v \abar \hspace{1cm}
		\Div_v &\bbar^m &\to \Div_v \bbar
	\end{alignat*}
	in $L^1_{loc}((0,T) \times \Rxv)$.

	\item \textbf{Uniform bound:}
	There exists a $C > 0$ such that, for every $m \in \mathbb{N}$,
	\begin{equation}
		\label{ch2:eq:approx-uniform-bound}
		\int_0^T \iint f_m (1 + |v|^2 + |x-tv|^2) \; dxdv dt \leq C
		\text{ and }
		0 \leq f_m \leq \varepsilon^{-1}.
	\end{equation}
	\end{enumerate}
	We call the sequence of approximating functions $(f_m)_m$ an \emph{approximation scheme}.
	The matrix $\abar^m$ is called the diffusion matrix and it is sometimes useful to explicitly state which sequence of diffusion matrices we are using.
	In this case, we say that $f_m \to f$ is an approximation scheme with diffusion $(\abar^m)_m$.
\end{definition}

Note that in the definition above, we do not assume any specific form for the coefficients $\abar^m$ and $\bbar^m$, which are treated here merely as sequences converging to the LFD coefficients.

The main theorem reads as follows.
\begin{theorem}
	\label{ch2:thm:main-theorem}
	Let $\varepsilon_n \to 0$.
	For each $n$, let $f^n$ be a suitable solution of the LFD equation with quantum parameter $\varepsilon_n$ and initial data $0 \leq f^n_0 \leq \varepsilon^{-1}_n$.
	For each $n$, let $f^n_m \to f^n$ be an approximation scheme with diffusion $(\abar^{n,m})_m$.

	Suppose the diffusion matrices satisfy
	\begin{equation}
		\label{ch2:eq:abar-m-elliptic}
		\abar^{n,m} \geq a_m *_v (f^n_m (1-\varepsilon_n f^n_m)),
		\text{ where } a_m(z) = \Gamma_m(|z|) \left( I - \frac{z \otimes z}{|z|^2} \right),
	\end{equation}
	with $\Gamma_m$ such that $\Gamma_m(|z|) \to \Gamma(|z|)$ in $L^r_{loc}(\mathbb{R}^N)$, for the $r$ in \eqref{ch2:property:collision-kernel-integrability}, and for every $R > 0$ there exists some $K_R > 0$ such that
	\[
		\Gamma_m(|z|) \geq K_R
		\;\; \forall |z| \leq R.
	\]

	If the sequence of initial data $(f^n_0)_n$ is such that
	\begin{equation}
		\label{ch2:eq:initial-data-convergence}
		\iint f^n_0 (1 + |x|^2 + |v|^2 + \log f^n_0) \, dxdv
		\to
		\iint f_0 (1 + |x|^2 + |v|^2 + \log f_0) \, dxdv,
	\end{equation}
	then, up to a subsequence, the sequence $(f^n)_n$ converges to a renormalized solution of the Landau equation with defect measure, with the quadratic term interpreted as in Remark \ref{ch2:rem:quadratic-term}.
\end{theorem}

Hypothesis \eqref{ch2:eq:abar-m-elliptic} arises as a requirement of the techniques we use, and also as a result of the fact that we would like to write the quadratic term of the renormalized Landau form in the ``natural'' form \eqref{ch2:eq:formula-for-quad-term}.

It is worth noting that the solutions constructed in \cite{sampaio-2024} are suitable weak solutions satisfying \eqref{ch2:eq:abar-m-elliptic}.
However, this convolution structure is not essential for proving most of the results, and if we are willing to offer an alternative, albeit more abstract, interpretation for the quadratic term, we can alleviate the requirement \eqref{ch2:eq:abar-m-elliptic}, which specifies the way in which we approximate the diffusion matrix $\abar$, in favor of a slightly more general ellipticity requirement.

\begin{theorem}
	\label{ch2:thm:alternative-entropy}
	Let $\varepsilon_n \to 0$.
	For each $\varepsilon_n$, let $f^n$ be a suitable solution of the LFD equation with quantum parameter $\varepsilon_n$ and initial data $f^n_0$.
	For each $n$, let $f^n_m \to f^n$ be an approximation scheme with diffusion $(\abar^{n,m})_m$.

	Suppose the collision kernel satisfies
	\begin{equation}
		\label{ch2:eq:collision-kernel-extra-regularity}
		\nabla_v a \in \mathcal{M}(\mathbb{R}^N) + L^\infty(\mathbb{R}^N),
	\end{equation}
	and the diffusion matrices $\nabla_v \abar^{n,m} \in L^\infty_{loc}((0,\infty) \times \Rxv)$ of the approximation scheme are such that, for every $n \in \mathbb{N}$,
	\begin{equation}
		\label{ch2:eq:nabla-v-a-nm-convergence}
		\nabla_v \abar^{n,m} \to \nabla_v \abar^n
		\quad \text{in} \quad
		L^1_{loc}((0,\infty) \times \Rxv).
	\end{equation}

	If for every $R, R'>0$, and $\eta \in \mathbb{R}^N$, we have that
	\begin{equation}
		\label{ch2:ineq-uniform-ellip}
		\abar^{n,m}_{ij} \eta_i \eta_j \geq \int_{|v_*| \leq R'} \nu
		\left[ 1 - \biggl( \frac{v-v_*}{|v-v_*|} \cdot \frac{\eta}{|\eta|} \biggr)^2 \right] |\eta|^2 f^n_{m,*} (1-f^n_{m,*}) \, dv_*,
	\end{equation}
	for every $v \in \mathbb{R}^N$ such that $|v| \leq R$, and the sequence of initial data satisfies \eqref{ch2:eq:initial-data-convergence},
	then, up to a subsequence, the sequence $f^n$ converges to a renormalized solution $f$ of the Landau equation with defect measure.
\end{theorem}

\begin{remark}
	\label{ch2:rem:abstract-quadratic-term}
	The definition of the quadratic term $- \beta''(f) \abar_{ij} \pdv{f}{v_i} \pdv{f}{v_j}$ in Theorem \ref{ch2:thm:alternative-entropy} differs from that in Theorem \ref{ch2:thm:main-theorem}.
	In the previous theorem, the convolution structure allowed us to identify the quadratic term explicitly, as in \eqref{ch2:eq:formula-for-quad-term}.
	In the context of Theorem \ref{ch2:thm:alternative-entropy}, we lack the convolution structure.
	Instead, the quadratic term is defined via a smooth approximation of the matrix square root, denoted by $S^k$.

	We define the distribution
	\[
		S^k(\abar) \nabla_v \gamma(f)
		\equiv \Div_v [ S^k(\abar) \gamma(f)]
		- \Div_v [S^k(\abar)] \gamma(f).
	\]
	We show that the sequence $(S^k(\abar) \nabla_v \gamma(f))_k$ converges weakly in $L^2_{loc}((0,\infty) \times \Rxv)$, up to extraction of a subsequence.
	The quadratic term of the renormalized Landau equation is then defined to be the square of the limit of this sequence.
\end{remark}

The technique we will use to prove Theorems \ref{ch2:thm:main-theorem} and \ref{ch2:thm:alternative-entropy} consists of using the extra regularity provided by the $(f^n_m)_n$ approximations to pass the equation into a renormalized form and then study the compactness of the diagonal, as illustrated in Figure \ref{ch2:diagram:diagonal}.

\begin{figure}[htp]
\begin{center}
  \includegraphics[width=0.22\textwidth]{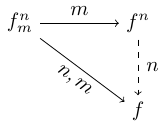}
  \caption{The diagonal limit. The solid arrows represent the limits or compactness properties available. The dashed arrow represents the limit we wish to establish.}\label{ch2:diagram:diagonal}
  \end{center}
\end{figure}

That is, in order to show the limit $f^n \to f$, we will first show that there exists a sequence $m_n \to \infty$ such that $(f^n_{m_n})_n$ is strongly compact in $L^1$.
Since the approximation schemes $(f^n_m)$ are always ``close'' to $f^n$, we can then show the convergence of $f^n$ ``in tow'' to the same limit, which we will show is a renormalized solution of the Landau equation with defect measure.

In order to do this, in Section \ref{ch2:sec:weak-compactness-diagonal} we start by proving that any time we have such an arrangement, there exists some diagonal that is weakly compact.
We then proceed in Section \ref{ch2:sec:weak-to-renormalized} to show that approximate solutions satisfy in fact a renormalized formulation of LFD, which we will use in Section \ref{ch2:sec:diagonal-compactness} to prove that our diagonal is in fact strongly compact.
Section \ref{ch2:sec:thm-1-proof} is then dedicated to using this compactness to pass the renormalized formulation to the limit, thus proving Theorem \ref{ch2:thm:main-theorem}, and finally, in Section \ref{ch2:sec:thm-2-proof} we prove Theorem \ref{ch2:thm:alternative-entropy} by studying approximations of the matrix square root.

\section{Weak compactness of diagonal sequences}
\label{ch2:sec:weak-compactness-diagonal}
Throughout this paper, we adopt the following abbreviated notation for integrals.
We denote by $\int_{xv}$ the integral over the entire phase space $\Rxv$.
Similarly, unless explicitly stated otherwise, $\int_{txv}$ denotes the space-time integral over $(0,\infty) \times \Rxv$.

In the existence theory for the Boltzmann, Landau, and Landau-Fermi-Dirac equations, the conservation laws obeyed by these equations automatically imply their solutions are weakly compact in $L^1_{loc}((0,\infty), L^1(\Rxv))$.

In our case, however, the approach needs to be slightly different, since our approximation schemes do not necessarily satisfy conservation laws.
In fact, this is neither assumed about the schemes nor can it be deduced from the approximating equations themselves, since we do not have a convolution structure for $\abar$ and $\bbar$ in these equations.
Despite this, we can construct a diagonal sequence that is weakly compact in $L^1$, and the main result of this section then reads as follows.
\begin{lemma}
	\label{ch2:lem:diag-weak-convergence}
	Let $f^n$ be a sequence of suitable weak solutions to the LFD equation \eqref{ch2:eq:lfd}, as in Definition \ref{ch2:def:suitable-weak-solution}, and for each $n$, consider $f^n_m \to f^n$ as an approximation scheme for $f^n$.

	If there exists a $C > 0$ such that
	\[
		\int_{xv} f^n_0 (1 + |x|^2 + |v|^2 + \log f^n_0) \, dxdv \leq C,
	\]
	then there exists a sequence $(M_n)_n$ such that if $m_n \geq M_n$ for every $n \in \mathbb{N}$, then the diagonal sequence $(f^n_{m_n})_n$ is weakly compact in $L^1_{loc}((0,\infty); L^1(\Rxv))$.
\end{lemma}

\begin{proof}
Throughout this proof, for every $t > 0$, let us denote by $\int_{0}^t \intxv$ the integral over $(0,t) \times \Rxv$.

The main idea is that although we do not know if the approximate solutions obey conservation laws, the $f^n$ solutions definitely do, so we choose a diagonal sequence $f^n_{m_n}$ that is close enough to the $f^n$ sequence.

\proofpart{I}{Bounding the approximating terms by their limits.}
We will start by showing that for every $T > 0$ and $n \in \mathbb{N}$, there is a $C_T > 0$ and an $M_{n,T} > 0$ such that
\begin{equation}
	\label{ch2:ineq:fnm-weighted-bounds-by-fn}
	\int_0^T \intxv f^n_m (1 + |v| + |x-tv| + |\log f^n_m|) \leq C_T
	+ \int_0^T \intxv f^n \log f^n,
\end{equation}
for every $m \geq M_{n,T}$.

Indeed, for every $n$, the sequence $(f^n_m (1 + |v|^2 + |x-tv|^2))_m$ is uniformly bounded in $L^1((0,T);\allowbreak L^1(\Rxv))$, from \eqref{ch2:eq:approx-uniform-bound}.
We also have that $f^n_m \to f^n$ in $L^1((0,T); L^1(\Rxv))$, hence it follows by interpolation that
\begin{equation}
	\label{ch2:fnm-convergence}
	f^n_m (1 + |v| + |x-tv|) \to f^n (1 + |v| + |x-tv|)
	\text{ in } L^1((0,T); L^1(\Rxv)),
\end{equation}
thus there exists some $M_{n,T}>0$ such that
\begin{align*}
	\int_0^T \intxv f^n_m (1 + |v| + |x-tv|)
	&\leq \int_0^T \intxv f^n (1 + |v| + |x-tv|) + 1\\
	&\leq 2\int_0^T \intxv f^n (1 + |v|^2 + |x-tv|^2) + 1.
\end{align*}
for every $m \geq M_{n,T}$.
Then, since $f^n$ is a solution of \eqref{ch2:eq:lfd} as in Definition \ref{ch2:def-weak-solution}, we have that
\begin{equation}
	\label{ch2:eq:quadratic-bounds-from-convergence}
	\int_0^T \intxv f^n (1 + |v|^2 + |x-tv|^2)
	\leq T \intxv f^n_0 (1 + |v|^2 + |x|^2),
\end{equation}
which then achieves the first part of inequality \eqref{ch2:ineq:fnm-weighted-bounds-by-fn}.

For the entropy part, notice that, for any function $h = h(t,x,v)$, we have
\[
	h |\log h|
	= h \log h - 2 h \log h \mathbbm{1}_{\{ h \leq 1 \}}.
\]
We can bound the last term as
\begin{align*}
	- h \log h \mathbbm{1}_{\{ h \leq 1 \}}
	&= h \log (1/h) \mathbbm{1}_{\{h \leq \exp(-|x| - |v|)\}}
	+ h \log (1/h) \mathbbm{1}_{\{ h > \exp(-|x| - |v|)\}}\\
	&\leq h \log (1/h) \mathbbm{1}_{\{h \leq \exp(-|x| - |v|)\}}
	+ h (|x| + |v|),
\end{align*}
and since $x \log (1/x) \leq \sqrt{x}$, for every $x \geq 0$, we conclude that
\[
	- h \log h \mathbbm{1}_{\{ h \leq 1 \}}
	\leq \exp\left( - (|x| + |v|)/2 \right)
	+ h (|x| + |v|),
\]
thus
\begin{equation}
	\label{ch2:ineq:h-logh}
	h |\log h|
	\leq 2\exp\left[ - \frac{1}{2} \big( |x| + |v| \big) \right]
		+ h \log h
		+ 2h \big( |x| + |v| \big).
\end{equation}
Then, choosing $h(t,x',v) = g(t,x'+tv,v)$ and letting $x = x'+tv$, we have that
\begin{equation}
\label{ch2:eq:fn-logfn}
\begin{split}
	g(t,x,v) |\log g(t,x,v)|
	&\leq 2\exp\left[ - \frac{1}{2} \big( |x-tv| + |v| \big) \right]\\
	&\quad + g(t,x,v) \log g(t,x,v)
	+ 2g(t,x,v) \big( |x-tv| + |v| \big).
\end{split}
\end{equation}
for every $(t,x,v) \in (0,T) \times \Rxv$.
So, in particular, for $0 \leq f^n_m \leq \varepsilon_n^{-1}$,
\[
	f^n_m |\log f^n_m|
		\leq 2\exp\left[ - \frac{1}{2} \big( |x-tv| + |v| \big) \right]
			+ (2 + \varepsilon_n^{-1})
			f^n_m \big( 1 + |x-tv| + |v| \big).
\]
Hence, from \eqref{ch2:fnm-convergence} it follows that for each $n$, we have that $f^n_m | \log f^n_m |$ converges to $f^n | \log f^n |$ in $ L^1((0,T); L^1(\Rxv))$  by the general dominated convergence theorem (see, for example, Section 4.4, Theorem 19 of \cite{royden_2010}).
There exists some $M_T$ such that
\[
	\int_0^T \intxv f^n_m |\log f^n_m|
	\leq \int_0^T \intxv f^n |\log f^n| + 1,
\]
for every $m \geq M_{n,T}$.
Then, applying inequality \eqref{ch2:eq:fn-logfn} for $g = f^n$ gives
\[
	\int_0^T \intxv f^n_m |\log f^n_m|
	\leq C_N + \int_0^T \intxv f^n \log f^n
	+ 2 \int_0^T \intxv f^n (|x-tv| + |v|),
\]
which implies inequality \eqref{ch2:ineq:fnm-weighted-bounds-by-fn} by the means of \eqref{ch2:eq:quadratic-bounds-from-convergence}.

\proofpart{II}{Entropy bound.}

We now pass to the entropy part of the proposition.
Let us show that
\begin{equation}
	\label{ch2:ineq-entropy-proof-1}
	\intxv f^n \log f^n \leq \intxv f^n_0 \log f^n_0 + \intxv f^n_0.
\end{equation}

Since $0 \leq f^n_0 \leq \varepsilon^{-1}$, we have that
\begin{align*}
	\intxv f^n_0 \log f^n_0
	&\geq \frac{1}{\varepsilon} \intxv \varepsilon f^n_0 \log f^n_0 + (1-\varepsilon f^n_0) \log(1-\varepsilon f^n_0)\\
	&= S_\varepsilon(f^n_0) - \intxv f^n_0 \log(\varepsilon)\\
	&\geq S_\varepsilon(f(t)) - \intxv f^n_0 \log(\varepsilon).
\end{align*}

Now, using that for every $0 \leq x \leq \varepsilon^{-1}$, we have
\[
	\frac{1}{\varepsilon} \left[ \varepsilon x \log(\varepsilon x) + (1-\varepsilon x) \log(1-\varepsilon x) \right]
	\geq x \log x + x \log \varepsilon - x,
\]
it follows that
\[
	\intxv f^n_0 \log f^n_0
	\geq \intxv f^n \log f^n
	+ \intxv f^n \log(\varepsilon)
	- \intxv f^n
	- \intxv f^n_0 \log(\varepsilon),
\]
and then \eqref{ch2:ineq-entropy-proof-1} follows from the conservation of mass of $f^n$.

\proofpart{III}{Weak compactness.}
	Now, applying \eqref{ch2:ineq:fnm-weighted-bounds-by-fn} together with \eqref{ch2:ineq-entropy-proof-1}, we have that for all $T>0$ and $n \in \mathbb{N}$, there exists a $C_T>0$ and an $M_{n,T}>0$ such that
	\[
		\int_0^T \intxv f^n_m (1 + |v| + |x-tv| + |\log f^n_m|) \leq C_T,
	\]
        for all $m \geq M_{n,T}$, which is indeed the first part of the statement.
        In this final step, we show that the bound we have shown indeed implies the weak compactness from the statement.

        Let $N_n = \max_{T\in\{1, \dots, n\}} M_{n,T}$ and suppose that the diagonal sequence $(f^n_{m_n})_n$ satisfies $m_n \geq N_n$ for every $n \in \mathbb{N}$.
        For every $n \geq 1$, we have that $m_n \geq N_n \geq M_{n,1}$, hence
        \begin{equation} \label{ch2:eq:fnmn-weights-1}
                \int_0^1 \intxv f^n_{m_n} (1 + |v|^2 + |x-tv|^2 + |\log f^n_{m_n}|) \leq C_1.
        \end{equation}
        Let us write $g^n$ instead of $f^n_{m_n}$.
        The bound \eqref{ch2:eq:fnmn-weights-1} gives us, from the Dunford-Pettis theorem, that there exists a sequence $(k_1(n))_n$ such that
	\[
		g^{k_1(n)} \weakto g^1 \text{ in } L^1((0,1) \times \Rxv).
	\]
        Next, for each $n \geq 2$, we have that $m_n \geq N_n \geq M_{n,2}$, hence
        \[
                \int_0^2 \intxv f^n_{m_n} (1 + |v|^2 + |x-tv|^2 + |\log f^n_{m_n}|) \leq C_2
        \]
        and similarly this implies there exists a $(k_2(n))_n$, a subsequence of $(k_1(n))_n$, such that
	\[
		g^{k_2(n)} \weakto g^2 \text{ in } L^1((0,2) \times \Rxv)
	\]
	and from the uniqueness of the weak limit, we have that $g^2 = g^1$ a.e. in $(0,1) \times \Rxv$.

	We repeat this argument, taking subsequences of subsequences to construct, for every $m$, a sequence $(g^{k_m(n)})_n$ that converges weakly in $L^1((0,m) \times \Rxv)$ to $g^m$.
	As before, we have that if $N < M$, then $g^N$ and $g^M$ coincide in $t \in (0,N)$.
	Thus, if we consider
	\[
		g = \sum_{n = 1}^\infty g^n \mathbbm{1}_{[n-1,n)}(t),
	\]
	then $g^N$ is the restriction of $g$ to $(0,N)$.

	The diagonal sequence $g^{k_n(n)}$ converges to $g$ weakly in $L^1_{loc}((0,\infty); L^1(\Rxv))$.
	Indeed, for every $T > 0$, let $N > T$.
	By construction, the sequence $g^{k_n(n)}$ is a subsequence of $g^{k_N(n)}$, for $n > N$, thus it converges to $g^N$ in $(0,N)$ and in particular it converges to $g$ in $(0,T)$.

	We then relabel the diagonal sequence $k_n(n)$ simply as $n$, and we have, to summarize,
	\[
		f^n_{m_n} \weakto g\; \text{ in } L^1_{loc}((0,\infty); L^1(\Rxv)).
	\]
	\end{proof}

\section{From weak to renormalized solutions}
\label{ch2:sec:weak-to-renormalized}
A crucial step in the proof of Theorem \ref{ch2:thm:main-theorem} consists in finding a renormalized formulation for the LFD equation, which we will later pass to the limit by making the quantum parameter tend to zero.
So far, it is not known whether weak solutions of the LFD are also renormalized solutions (and in what sense these renormalized solutions would be defined), so in this section we will show that if we had a little more regularity in the variable $v$, then weak solutions are renormalized and vice versa.

\begin{proposition}
	\label{ch2:prop:approxim-implies-renorm}
	Suppose $g \in L^2_{loc}((0,\infty) \times \mathbb{R}^N_x; H^1_{loc}(\mathbb{R}^N_v)) \cap L^\infty((0,\infty) \times \Rxv)$. Let $\beta_\delta(t) \equiv \frac{t}{1+\delta t}$ and $\Abar_{ij}, \Bbar_i, \pdv{\Bbar_i}{v_i} \in L^\infty_{loc}$.
	Thus, $g$ is a solution to
	\begin{equation}
		\label{ch2:eq:prop-1-weak-form}
		\pdv{g}{t} + v_i \pdv{g}{x_i} =
		\pdv{}{v_j} \left\{ \Abar_{ij} \pdv{g}{v_i} - \Bbar_i g (1 - \varepsilon g) \right\}
	\end{equation}
	in $\mathcal{D}'((0,\infty) \times \Rxv)$ if and only if $\beta_\delta(g)$ is a solution to
	\begin{equation}
		\label{ch2:eq:g-renormalized}
		\begin{split}
		\pdv{\beta_\delta(g)}{t} + v\cdot\nabla_x \beta_\delta(g)
		&= \pdv{}{v_i}
		\left\{
			\Abar_{ij} \pdv{\beta_\delta(g)}{v_j}
			-\Bbar_i (1 - 2\varepsilon g) \beta_\delta(g)
			- 2 \varepsilon \Bbar_i B_\delta(g)
		\right\}\\
		&\quad -\beta_\delta''(g) \Abar_{ij} \pdv{g}{v_i} \pdv{g}{v_j}\\
		&\quad -\Cbar \left[
			g (1 - \varepsilon g)\beta_\delta'(g)
			- (1 - 2\varepsilon g)\beta_\delta(g)
			- 2\varepsilon B_\delta(g)
		\right]
		\end{split}
	\end{equation}
	in $\mathcal{D}'((0,\infty) \times \Rxv)$, where $\Cbar = \sum_i \pdv{\Bbar_i}{v_i}$ and $B_\delta(x) = \int_{0}^{x} \beta_\delta(t) \; dt$.
\end{proposition}
\begin{proof}
	We begin with the direct implication.
	The formal procedure consists of multiplying both sides of the equation by $\beta_\delta'$ and then rewriting the derivatives in the appropriate form, using the product rule. However, due to the low regularity of $g$ (particularly with respect to the variables $t$ and $x$), we cannot justify this procedure directly.

	To overcome this, take $\rho$ to be a $C^\infty(\Rxv)$ function with support in $(-1, 0) \times B_x(0, 1)$, and define the mollifying sequence
	\[
		\rho_m(t,x) = \frac{1}{m^{N+1}} \rho\left(
			\frac{t}{m}, \frac{x}{m} \right).
	\]
	Fix $t > 0$ and $x \in \mathbb{R}^N_x$ and consider a test function $\varphi = \varphi(v)$ in $C^\infty_c(\mathbb{R}^N_v)$.
	Testing equation \eqref{ch2:eq:prop-1-weak-form} against
	\[
		\Psi_{t,x} =
		\begin{cases}
			\mathbb{R}_t \times \mathbb{R}^N_x &\to \mathbb{R}\\
			\hfil (s,y) &\mapsto \rho_m(t-s, x-y) \varphi(v)
		\end{cases}
	\]
	leads to
	\begin{equation*}
	\begin{split}
		& - \int_0^\infty \iint g(s,y,v) \left( \pdv{\rho_m}{t} + v_i \pdv{\rho_m}{x_i} \right)(t-s, x-y,v) \varphi(v) \; dydv ds\\
		&  = \int_0^\infty \iint \left[ \Abar_{ij} \pdv{g}{v_j} - \Bbar_i g (1-\varepsilon g) \right](s,y,v)
			\rho_m(t-s, x-y) \pdv{\varphi}{v_i}(v) \; dydvds,
	\end{split}
	\end{equation*}
	which are convolutions (in the variables $t$ and $x$) with $\rho_m$.
	Here, it will be useful to establish a notation for the rest of the proof. For any distribution $T$, let us denote by $T_m$ the convolution
	\[
		T_m \equiv T *_{t,x} \rho_m.
	\]
	This way, we rewrite the above expression as
	\begin{align*}
	\begin{split}
		&- \int \pdv{g_m}{t}(t,x,v) \varphi(v) \; dv
		- \int v_i \pdv{g_m}{x_i}(t,x,v) \varphi(v) \; dv \\
		&  = \int \left[ \Abar_{ij} \pdv{g}{v_j} - \Bbar_i g (1-\varepsilon g) \right]_m (t,x,v) \pdv{\varphi}{v_i}(v) \; dv,
	\end{split}
	\end{align*}
	for every $t>0$ and $x \in \mathbb{R}^N_x$.

	Let $\varphi \in C^\infty_c((0, \infty) \times \Rxv)$.
	Taking $v \mapsto \varphi(t,x,v)$ as a test function in the expression above and integrating in $(t,x) \in (0, +\infty) \times \mathbb{R}^N$, we obtain
	\[
		- \inttxv \pdv{g_m}{t} \varphi
		- \inttxv v_i \pdv{g_m}{x_i} \varphi\\
		= \inttxv \left[ \Abar_{ij} \pdv{g}{v_j} - \Bbar_i g (1-\varepsilon g) \right]_m \pdv{\varphi}{v_i}.
	\]

	Now, let $\chi_k$ be a mollifying sequence in all three variables $(t,x,v)$, with support in $(-1, 0) \times B_x(0, 1) \times B_v(0,1)$.
	Taking $(\beta_\delta'(g_m) * \chi_k) \varphi$ as a test function in the equation above, we have
	\begin{equation}
		\label{ch2:eq:gm-renormalization-with-k}
		\begin{split}
		&- \inttxv \pdv{g_m}{t}(\beta_\delta'(g_m) * \chi_k) \varphi
		- \inttxv v_i \pdv{g_m}{x_i} (\beta_\delta'(g_m) * \chi_k) \varphi\\
		&  = \inttxv \left[ \Abar_{ij} \pdv{g}{v_j} - \Bbar_i g (1-\varepsilon g) \right]_m
			\pdv{[(\beta_\delta'(g_m) * \chi_k) \varphi]}{v_i}.
		\end{split}
	\end{equation}

	We start by passing this equation to the limit $k \to \infty$.
	To lighten the notation, for the rest of this proof, let us simply write $L^p_{loc}$ for $L^p_{loc}((0,\infty) \times \Rxv)$.
	Since $|\beta_\delta'(g_m)| \leq 1$, we have, by standard convolution results,
	\begin{equation*}
	\begin{split}
		\pdv{g_m}{t} (\beta_\delta'(g_m) * \chi_k)
		&\xrightarrow{k \to \infty}
		\pdv{g_m}{t} \beta_\delta'(g_m)
		= \pdv{\beta_\delta(g_m)}{t}\\
		v_i \pdv{g_m}{x_i} (\beta_\delta'(g_m) * \chi_k)
		&\xrightarrow{k \to \infty}
		v_i \pdv{g_m}{x_i} \beta_\delta'(g_m)
		= v_i \pdv{\beta_\delta(g_m)}{x_i}
	\end{split}
	\end{equation*}
	in $L^p_{loc}$, for every $1 \leq p < \infty$,
	hence the left-hand side of \eqref{ch2:eq:gm-renormalization-with-k} converges to
	\[
		- \inttxv \pdv{\beta_\delta (g_m)}{t} \varphi
		- \inttxv v_i \pdv{\beta_\delta(g_m)}{x_i} \varphi
		=
		\inttxv \beta_\delta(g_m) \pdv{\varphi}{t}
		+ \inttxv \beta_\delta(g_m) v_i \pdv{\varphi}{x_i},
	\]
	after integration by parts.
	For the right-hand side of \eqref{ch2:eq:gm-renormalization-with-k}, consider
	\[
		\pdv{[(\beta_\delta'(g_m) * \chi_k) \varphi]}{v_i} =
			\left[ \left(\beta_\delta''(g_m) \pdv{g_m}{v_i}\right) * \chi_k \right] \varphi
			+ \left( \beta_\delta'(g_m) * \chi_k \right) \pdv{\varphi}{v_i}.
	\]
	Using that $\beta_\delta''(g_m) \pdv{g_m}{v_i}$ and $\beta_\delta'(g_m)$ are in $L^2_{loc}$, the above expression converges to
	\[
		\beta_\delta''(g_m) \pdv{g_m}{v_i} \varphi	+ \beta_\delta'(g_m) \pdv{\varphi}{v_i}
	\]
	in $L^2_{loc}$ and is compactly supported.
	Since $\left[ \Abar_{ij} \pdv{g}{v_j} - \Bbar_i g (1-\varepsilon g) \right]_m \in L^2_{loc}$, passing \eqref{ch2:eq:gm-renormalization-with-k} to the limit $k \to \infty$ gives
	\begin{equation}
		\label{ch2:eq:gm-renormalization}
		\begin{split}
		&\inttxv \beta_\delta(g_m) \pdv{\varphi}{t}
		+ \inttxv \beta_\delta(g_m) v_i \pdv{\varphi}{x_i} \\
		&= \inttxv \left[ \Abar_{ij} \pdv{g}{v_j} - \Bbar_i g (1-\varepsilon g) \right]_m
			\left(\beta_\delta''(g_m) \pdv{g_m}{v_i} \varphi	+ \beta_\delta'(g_m) \pdv{\varphi}{v_i}\right),
		\end{split}
	\end{equation}
	and it remains to pass this equation to the limit $m \to \infty$.

	The left-hand side of \eqref{ch2:eq:gm-renormalization} converges to
	\[
		\inttxv \beta_\delta(g) \pdv{\varphi}{t}
		+ \inttxv \beta_\delta(g) v_i \pdv{\varphi}{x_i}
	\]
	since $g_m \to g$ almost everywhere and $|\beta_\delta(g_m)| \leq \delta^{-1}$.
	For the integral
	\[
		\inttxv \left[ \Abar_{ij} \pdv{g}{v_j} \right]_m
			\left(\beta_\delta''(g_m) \pdv{g_m}{v_i} \varphi + \beta_\delta'(g_m) \pdv{\varphi}{v_i} \right),
	\]
	we have that $\left[ \Abar_{ij} \pdv{g}{v_j} \right]_m \to \Abar_{ij} \pdv{g}{v_j}$ in $L^2_{loc}$.
	Passing to a subsequence, if necessary, we have that $\pdv{g_m}{v_i}$ is dominated by an $L^2_{loc}$ function.
	Since $\beta_\delta''$ is bounded, we conclude by dominated convergence that
	\[
		\left[ \Abar_{ij} \pdv{g}{v_j} \right]_m \beta_\delta''(g_m) \pdv{g_m}{v_i}
		\to  \beta_\delta''(g) \Abar_{ij} \pdv{g}{v_i} \pdv{g}{v_j}
	\]
	and
	\[
		\left[ \Abar_{ij} \pdv{g}{v_j} \right]_m \beta_\delta'(g_m)
		\to \Abar_{ij} \pdv{g}{v_j} \beta_\delta'(g) = \Abar_{ij} \pdv{\beta_\delta(g)}{v_j}
	\]
	in $L^1_{loc}$.
	Thus, the above integrals converge to
	\[
		\inttxv \beta_\delta''(g) \Abar_{ij} \pdv{g}{v_i} \pdv{g}{v_j} \varphi
		+ \inttxv \Abar_{ij} \pdv{\beta_\delta(g)}{v_j} \pdv{\varphi}{v_i}
	\]
	and, similarly, one can prove that
	\begin{equation*}
	\begin{split}
		&\inttxv \left[ \Bbar_i g (1-\varepsilon g) \right]_m
		\left(\beta_\delta''(g_m) \pdv{g_m}{v_i} \varphi	+ \beta_\delta'(g_m) \pdv{\varphi}{v_i}\right) \\
		&\to
		\inttxv \Bbar_i g (1-\varepsilon g)
		\left(\beta_\delta''(g) \pdv{g}{v_i} \varphi + \beta_\delta'(g) \pdv{\varphi}{v_i}\right).
	\end{split}
	\end{equation*}

	However, the limit distribution we have above is not yet in the same form as that found in equation \eqref{ch2:eq:g-renormalized}.
	Then, consider $(g_n)_n$ a sequence of $C^\infty_c((0,\infty) \times \Rxv)$ functions such that $g_n \to g$ in $L^2_{loc}((0,\infty) \times \mathbb{R}^N_x; H^1_{loc}(\mathbb{R}^N_v))$.
	For each $n$, we have
	\begin{equation*}
	\begin{split}
		&\beta_\delta'(g_n) \pdv{}{v_i} \left[ \Bbar_i g_n(1-\varepsilon g_n) \right]
\\&=\pdv{}{v_i}
		\left[
			\Bbar_i (1 - 2\varepsilon g_n) \beta_\delta(g_n)
			+ 2\varepsilon \Bbar_i B_\delta(g_n)
		\right]\\
		&\quad+\Cbar \left[
			g_n(1 - \varepsilon g_n)\beta_\delta'(g_n)
			- (1 - 2\varepsilon g_n)\beta_\delta(g_n)
			- 2\varepsilon B_\delta(g_n)
		\right],
	\end{split}
	\end{equation*}
	but also from the product rule,
	\begin{equation*}
	\begin{split}
		\beta_\delta'(g_n) \pdv{}{v_i} \left[ \Bbar_i g_n(1-\varepsilon g_n) \right]
		=
		&\pdv{}{v_i} \big[ \Bbar_i g_n(1-\varepsilon g_n) \beta_\delta'(g_n) \big]\\
		&  - \Bbar_i g_n(1-\varepsilon g_n) \beta_\delta''(g_n) \pdv{g_n}{v_i}.
	\end{split}
	\end{equation*}
	Passing both equations to the limit in the sense of distributions, we conclude that
	\begin{equation*}
	\begin{split}
		&- \inttxv \Bbar_i g (1-\varepsilon g)\left(\beta_\delta''(g) \pdv{g}{v_i} \varphi + \beta_\delta'(g) \pdv{\varphi}{v_i}\right)\\
		&= - \inttxv \left[
			\Bbar_i (1 - 2\varepsilon g) \beta_\delta(g)
			+ 2\varepsilon \Bbar_i B_\delta(g)
		\right] \pdv{\varphi}{v_i}\\
		&\quad +\inttxv \Cbar \left[
			g(1 - \varepsilon g)\beta_\delta'(g)
			- (1 - 2\varepsilon g)\beta_\delta(g)
			- 2\varepsilon B_\delta(g)
		\right] \varphi,
	\end{split}
	\end{equation*}
	for every $\varphi \in C^\infty_c((0,\infty) \times \Rxv)$, and therefore
	\begin{equation*}
	\begin{split}
		&- \inttxv \beta_\delta(g) \pdv{\varphi}{t}
		- \inttxv \beta_\delta(g) v \cdot \nabla_x \varphi \\
		&  = \inttxv \Abar_{ij} \beta_\delta(g) \mpdv{\varphi}{v_i}{v_j}\\
		&\quad + \inttxv \left[
			\pdv{\Abar_{ij}}{v_j} \beta_\delta(g)
			+ \Bbar_i (1-2\varepsilon g) \beta_\delta(g)
			+ 2 \varepsilon \Bbar_i B_\delta(g)
			\right] \pdv{\varphi}{v_i} \\
		&\quad  - \inttxv \beta_\delta''(g) \Abar_{ij} \pdv{g}{v_i} \pdv{g}{v_j} \varphi \\
		&\quad  + \inttxv \Cbar \left[
			g (1 - \varepsilon g)\beta_\delta'(g)
			- (1 - 2\varepsilon g)\beta_\delta(g)
			- 2\varepsilon B_\delta(g)
			\right] \varphi,
	\end{split}
	\end{equation*}
	for every $\varphi \in C^\infty_c((0,\infty) \times \Rxv)$.

	For the reverse implication, if $g$ satisfies the above equation for every $\delta > 0$, taking $\delta_n = 1/n \to 0$ and denoting $\beta_n$ for $\beta_{\delta_n}$ and $B_n$ for $B_{\delta_n}$, we have that, almost everywhere in $(0,\infty) \times \Rxv$,
	\[
		\beta_n(g) \to g,\;
		\beta_n'(g) \to 1,\;
		\beta_n''(g) \to 0
	\]
	and from dominated convergence, $B_n(g) \to \frac{g^2}{2}$ almost everywhere.
	This implies, given that $g$ is bounded, that the left-hand side converges to
	\[
		- \inttxv g \pdv{\varphi}{t}
		- \inttxv g v \cdot \nabla_x \varphi.
	\]
	For the right-hand side, we have that
	\begin{align*}
		\Abar_{ij} \beta_n(g)
			&\to \Abar_{ij} g\\
		\pdv{\Abar_{ij}}{v_j} \beta_n(g)
			+ \Bbar_i (1-2\varepsilon g) g
			+ 2 \varepsilon \Bbar_i B_n(g)
			&\to
			\pdv{\Abar_{ij}}{v_j} g
			+ \Bbar_i (1-\varepsilon g) g\\
		\beta_n''(g) \Abar_{ij} \pdv{g}{v_i} \pdv{g}{v_j}
			&\to 0\\
		\Cbar \left[
			g (1 - \varepsilon g)\beta_n'(g)
			- (1 - 2\varepsilon g)\beta_n(g)
			- 2\varepsilon B_n(g)
		\right]
			&\to 0
	\end{align*}
	almost everywhere, but since $|\beta_n(g)| \leq g$, $|\beta'_n(g)| \leq 1$, $|\beta''_n(g)| \leq 2$, $|B_n(g)| \leq g^2/2$ and $g \in L^\infty((0,\infty) \times \Rxv)$, these convergences actually hold in $L^1_{loc}$.
\end{proof}

An immediate corollary of the above result is that although we do not know whether weak LFD solutions are also renormalized solutions, the solutions in an approximation scheme definitely are.

\begin{lemma}
Let $f$ be a suitable solution of LFD, with $f_m \to f$ an approximation scheme, as in Definition \ref{ch2:def:suitable-weak-solution}.
Then, with $\beta_\delta(t) = \frac{t}{1+\delta t}$, we have, for each $n$,
\begin{equation}
	\label{ch2:lfd-renormalized}
	\begin{split}
	&\pdv{\beta_\delta(f_m)}{t}{} + v\cdot \nabla_x \beta_\delta(f_m) \\
	&  = \pdv{}{v_i}
	\biggl\{
		\abar^m_{ij} \pdv{\beta_\delta(f_m)}{v_j}
		-\bbar^m_i (1 - 2\varepsilon f_m) \beta_\delta(f_m)
		- 2 \varepsilon \bbar^m_i B_\delta(f_m)
	\biggr\}\\
	&\quad  -\beta_\delta''(f_m) \abar^m_{ij} \pdv{f_m}{v_i} \pdv{f_m}{v_j} \\
	&\quad -\cbar^m \biggl[
		f_m (1 - \varepsilon f_m)\beta_\delta'(f_m)
		- (1 - 2\varepsilon f_m)\beta_\delta(f_m)
		- 2\varepsilon B_\delta(f_m)
	\biggr]
	\end{split}
\end{equation}
in $\mathcal{D}'([0,T] \times \Rxv)$, where $\abar^m$ and $\bbar^m$ are the same as in equation \eqref{ch2:regularized} and $\cbar^m = \pdv{\bbar^m_i}{v_i}$.
\end{lemma}

\section{Diagonal compactness of approximation schemes}
\label{ch2:sec:diagonal-compactness}
The existence results for the Landau equation in \cite{villani_1996} and Landau-Fermi-Dirac in \cite{sampaio-2024} are, in essence, compactness results for approximation schemes.
For example, in the Landau-Fermi-Dirac case, one considers a sequence of solutions to an approximate equation belonging to the space $L^2_{loc}((0,\infty); L^2(\mathbb{R}^N_x; H^1(\mathbb{R}^N_v)))$.
It is shown that this sequence is strongly compact in $L^1_{loc}((0,\infty);\allowbreak L^1(\Rxv))$, which allows us to pass the approximate equation to the limit.

The idea then is to take advantage of this compactness of the approximation schemes to extract the compactness we want for a semi-classical limit.
However, since the LFD existence theorem is proved for a fixed quantum parameter, we cannot use it directly for a semi-classical limit.

Consider then a sequence of vanishing quantum parameters, $\delta_n \to 0$, and, for each $n$, let $f^n$ be a suitable solution of LFD, with $f^n_m \to f^n$ being an approximation scheme for this solution.
We want to prove that the diagonal of Figure \ref{ch2:diagram:diagonal} is strongly compact, that is, given a sequence $m_n \to \infty$, the sequence of approximate solutions $(f^n_{m_n})_n$ is strongly compact in $L^1$.
Once this diagonal has been established, we can choose $f^n_m$ closer and closer to $f^n$, and in the end we will have that the $f^n$ converge ``in tow'', culminating in the following lemma.
\begin{lemma}
	\label{ch2:lem:ae-comp}
	Let $(f^n)$ be a sequence of suitable weak solutions to \eqref{ch2:eq:lfd}.
	For each $n$, consider an approximation scheme $f^n_m \to f^n$ satisfying \eqref{ch2:ineq-uniform-ellip}.
	If the diagonal sequence $(f^n_{m_n})_n$ is weakly compact in $L^1_{loc}((0,\infty); L^1(\Rxv))$ and
	\[
		\abar^{n,m_n}, \;
		\Div_v \abar^{n,m_n}, \;
		\bbar^{n,m_n}, \;
		\cbar^{n,m_n}
	\]
	are bounded sequences in $L^1_{loc}((0,\infty) \times \Rxv)$,
	then the sequence $(f^n_{m_n})_n$ is strongly compact in the space $L^1_{loc}((0,\infty); L^1(\Rxv))$.
\end{lemma}

The outline of the proof of Lemma \ref{ch2:lem:ae-comp} is very reminiscent of the proof of the existence of weak solutions to the LFD equation.
First, we prove the compactness of velocity averages of the solution sequence, which gives us compactness in the $t$ and $x$ variables.
Then, we study the parabolic part of the equation in order to get compactness in the $v$ variable, and finally a proposition will allow us to join the two results and get full compactness in $L^1$ for these solutions.

\subsection{Velocity averaging}

In the study of transport equations, the averaging lemmas constitute a whole body of results stating roughly that, if $f$ is a solution to a transport equation $\partial_t f + v \cdot \nabla f = g$, then the velocity averages of $f$ have more regularity than $f$ itself.

For example, a classical result (see, for example, \cite{bouchut-2000}) states that if $f$ and $g$ are bounded in $L^2(\mathbb{R}_t \times \Rxv)$, then, for every $\varphi \in C^\infty_c(\mathbb{R}^N_v)$, the velocity average
\[
	\int f \varphi \; dv,
\]
is bounded in $H^{1/2}(\mathbb{R}_t \times \mathbb{R}^N_x)$.
Hence, if a sequence $f^n$ is uniformly bounded in $L^2$, the corresponding velocity averages are expected to be compact in $L^2$.

First shown independently by Agoshkov \cite{agoshkov_1984} and Golse, Perthame and Sentis in \cite{golse_1988} in an $L^2$ setting, averaging lemmas are now a standard technique in the study of kinetic equations.
The version we will use here is a simple corollary of Theorem 1.1.8 of \cite{bouchut-2000}, which reads as follows.
\begin{theorem}
	\label{ch2:bouchut-vel-avg}
	Let $\Omega$ be an open set of $\mathbb{R}_t \times \mathbb{R}^N_x$, and $(f_n)_n$ a sequence of functions bounded in $L^p_{loc}(\Omega \times \mathbb{R}^M)$, with $p>1$, verifying
	\[
		\partial_t f_n + v \cdot \nabla_x f_n =
		\sum_{|\alpha| \leq m} \partial^\alpha_v g_n^{(\alpha)}
		\text{ in }
		\Omega \times \mathbb{R}^M
	\]
	for some $m \in \mathbb{N}$, where $\big(g_n^{(\alpha)}\big)_n$ is a bounded sequence in the space of measures $\mathcal{M}_{loc}(\Omega \times \mathbb{R}^M)$ for each $|\alpha| \leq m$.
	For any $\psi \in C^\infty_c(\mathbb{R}^M)$, consider the velocity averages
	\[
		\rho^n_\psi(t,x) = \int f_n(t,x,v) \psi(v) \, dv.
	\]
	Then, the sequence $(\rho^n_\psi)_n$ is compact in $L^q_{loc}(\Omega)$ for any $q < p$.
\end{theorem}

In order to apply this theorem to the renormalized equation \eqref{ch2:lfd-renormalized} we   deduced in the last section, we first show an estimate for the quadratic term in the derivatives on the right-hand side of this equation.

\begin{lemma}
	\label{ch2:lem-aij-v-bounded}
	Let $(f^n)$ be a sequence of suitable weak solutions to \eqref{ch2:eq:lfd}.
	For each $n$, consider an approximation scheme $f^n_m \to f^n$, as in Definition \ref{ch2:def:suitable-weak-solution}.
	If the sequences
	\[
		\abar^{n,m_n}, \;
		\Div_v \abar^{n,m_n}, \;
		\bbar^{n,m_n}, \;
		\cbar^{n,m_n}
	\]
	are uniformly bounded in $L^1_{loc}((0,\infty) \times \Rxv)$, then for every $T, R > 0$, we have
	\[
		\iiint_{(0,T) \times B_R \times B_R}
		- \beta_{\delta}''(f^n_{m_n}) \abar^{n,m_n}_{ij} \pdv{f^n_{m_n}}{v_i} \pdv{f^n_{m_n}}{v_j} \, dtdxdv
		\leq C_{T,R}
	\]
	and also
	\[
		\iiint_{(0,T) \times B_R \times B_R} \abar^{n,m_n}_{ij} \pdv{\beta_\delta(f^n_{m_n})}{v_i} \pdv{\beta_\delta(f^n_{m_n})}{v_j} \, dtdxdv
		\leq C'_{T,R},
	\]
	where $C_{T,R}, C'_{T,R}>0$ are universal constants, which do not depend on $n$.
\end{lemma}
\begin{proof}
	In equation \eqref{ch2:lfd-renormalized}, choose $\beta_\delta(t) = \frac{t}{1+\delta t}$ as non-linearity and $\varphi \geq 0$ such that $\varphi \equiv 1$ in $(0,T) \times B_R \times B_R$ as a test function.
	This way, we have,
	\begin{align*}
		& \inttxv - \beta_{\delta}''(f^n_m)  \abar^{n,m}_{ij} \frac{\delta}{(1+\delta f^n_m)^3} \pdv{f^n_m}{v_i} \pdv{f^n_m}{v_j} \varphi \\
		& \leq \left| \intxv \beta_\delta(f^n_m)(0) \varphi(0) \right|
		+ \left| \inttxv \beta_\delta(f^n_m) \pdv{\varphi}{t} \right|\\
		&\quad + \left| \inttxv \beta_\delta(f^n_m) v \cdot \nabla_x \varphi \right|
		+ \left| \inttxv \abar^{n,m}_{ij} \beta_\delta(f^n_m) \mpdv{\varphi}{v_i}{v_j} \right|\\
		&\quad  + \left| \inttxv \left[
			\pdv{\abar^{n,m}_{ij}}{v_j} \beta_\delta(f^n_m)
			+ \bbar^{m,n}_i (1-2\varepsilon_n f^n_m) \beta_\delta(f^n_m)
			+ 2 \varepsilon_n \bbar^{m,n} B_\delta(f^n_m)
			\right] \pdv{\varphi}{v_i} \right|\\
		&\quad  + \left| \inttxv \cbar^{n,m} \left[
			f^n_m (1 - \varepsilon_n f^n_m)\beta_\delta'(f^n_m)
			- (1 - 2\varepsilon_n f^n_m)\beta_\delta(f^n_m)
			- 2\varepsilon_n B_\delta(f^n_m)
			\right] \varphi \right|.
	\end{align*}

		Since $|\beta_\delta(t)| \leq \delta^{-1}$, the first four terms on the right-hand side are bounded, up to multiplication by $\delta^{-1}$, by
	\[
		\| \varphi(0) \|_{L^1(\Rxv)}
		+ \left\| \pdv{\varphi}{t} \right\|_{L^1_{txv}}
		+ \left\| v \cdot \nabla_x \varphi \right\|_{L^1_{txv}}
		+ \big\| \abar^{n,m}_{ij} \big\|_{L^1(\supp \varphi)}
			\left\| \mpdv{\varphi}{v_i}{v_j} \right\|_{L^\infty_{txv}},
	\]
	where we have denoted by $L^p_{txv}$ the $L^p$-norm over the space-time domain $(0,\infty) \times \Rxv$.
	Now, since $0 \leq f^n_m \leq \varepsilon_n^{-1}$, we have $|1-2\varepsilon_n f^n_m| \leq 1$, and since $B_\delta(t) = \delta^{-2}(\delta t - \log(1 + \delta t))$, it follows that
	\[
		|\varepsilon_n B_\delta(f^n_m)|
		= \varepsilon_n \delta^{-2}
		\left[
			\delta f^n_m - \log(1+ \delta f^n_m)
		\right]
		\leq \delta^{-1} \varepsilon_n f^n_m
		\leq \delta^{-1}.
	\]
	Hence, the fifth term is bounded by
	\[
		\delta^{-1}
		\left(
			\left\| \Div_v \abar^{n,m} \right\|_{L^1(\supp\varphi)}
			+ 2\big\| \bbar^{n,m}_i \big\|_{L^1(\supp\varphi)}
		\right)
		\left\| \varphi \right\|_{L^\infty_{txv}}.
	\]

	Finally, using that $0 \leq t\beta_\delta'(t) = \frac{t}{(1+\delta t)^2} \leq \delta^{-1}$ for $t \geq 0$ and that $|1 - \varepsilon_n f^n_m| \leq 1$, we bound the last term by
	\[
		4 \delta^{-1} \big\| \cbar^{n,m} \big\|_{L^1(\supp\varphi)} \| \varphi \|_{L^\infty_{txv}},
	\]
	which proves the first bound.

	For the second bound, it suffices to notice that, for every $n, m \in \mathbb{N}$, we have
	\[
		\abar^{n,m}_{ij} \pdv{\beta_\delta(f^n_m)}{v_i} \pdv{\beta_\delta(f^n_m)}{v_j}
		= (\beta_\delta'(f^n_m))^2 \pdv{f^n_m}{v_i} \pdv{f^n_m}{v_j}.
	\]
	The result then follows since, for every $t > 0$,
	\[
		(\beta_\delta'(t))^2
		= \frac{1}{(1+\delta t)^4}
		\leq \frac{1}{(1+\delta t)^3}
		= - 2\delta^{-1} \beta_{\delta}''(f^n_m).
	\]
\end{proof}

We could use Lemma \ref{ch2:lem-aij-v-bounded} directly, together with the estimates we have for the $f^n$ solutions, to deduce the compactness of the velocity averages of $\beta(f^n)$.
However, as in the proof for the existence of global solutions to the LFD equation, it turns out that the quantity for which we must prove compactness for the velocity averages is $f^n (1-\varepsilon_n f^n)$.
Therefore, we need to do some extra work to get the compactness from one to the other.

\begin{lemma}
	\label{ch2:lem:velocity-averages-converge}
	Let $\varepsilon_n \to 0$ be a vanishing sequence of quantum parameters and for each $n$, and let $f^n$ be a suitable solution of LFD with quantum parameter $\varepsilon_n$, with $f^n_m \to f^n$ an approximation scheme, as in Definition \ref{ch2:def:suitable-weak-solution}.
	Define $F^n_m \equiv f^n_m (1-\varepsilon_n f^n_m)$.
	If $m_n \to \infty$ is such that the sequences
	\[
		\abar^{n,m_n}, \;
		\Div_v \abar^{n,m_n}, \;
		\bbar^{n,m_n}, \;
		\cbar^{n,m_n}
	\]
	are uniformly bounded in $L^1_{loc}((0,\infty) \times \Rxv)$ and the diagonal sequence $(f^n_{m_n})_n$ is weakly compact in $L^1_{loc}((0,\infty); L^1(\Rxv))$, then, for every $\delta > 0$, the velocity averages
	\[
		\int F^n_{m_n} \varphi \; dv\;
		\text{ and }
		\int \beta_{\delta}(f^n_{m_n}) \varphi \; dv\;
	\]
	are compact in $L^1_{loc}((0,\infty) \times \mathbb{R}^N_x)$, for every $\varphi \in C^\infty_c(\mathbb{R}^N_v)$.
\end{lemma}

\begin{proof}
	Fix $T > 0$, a compact set $K \subset \mathbb{R}^N_x$, and a test function $\varphi \in C^\infty_c(\mathbb{R}^N_v)$.
	Throughout this proof, we denote by $\int_{txv}$ and $\int_{tx}$ the integrations over $(0,T) \times K \times \mathbb{R}^N_v$ and $(0,T) \times K$, while $\int_v$ denotes integration over $\mathbb{R}^N_v$.
	We will also denote the sequence $f^n_{m_n}$ as $f^n$, as well as the coefficients $\abar^n_{m}$, $\bbar^n_m$, and $\cbar^n_m$ as $\abar^n$, $\bbar^n$, and $\cbar^n$, respectively.

	This proof is divided into three parts.
	First, we will show the compactness of the velocity averages of $\beta_\delta(f^n)$, using the renormalized formulation.
	We then proceed by showing that the velocity averages of $\beta_\delta(f^n)$ converge to the same limit as the velocity averages of $\beta_\delta(F^n)$, and finally we
	relax the value of $\delta > 0$ to achieve compactness for the velocity averages of $F^n$.

	\proofpart{I}{Velocity averages of $\beta_\delta(f^n)$.}
	Equation \eqref{ch2:lfd-renormalized} can be rewritten as
	\[
		\pdv{\beta_\delta(f^n)}{t} + v\cdot\nabla_x \beta_\delta(f^n)
		= \mpdv{G^n_{ij}}{v_i}{v_j}
		- \pdv{G_i^n}{v_i}
		+ G^n,
	\]
	where the functions on the right-hand side are given by
	\begin{align*}
		G_{ij}^n &= \abar^n_{ij} \beta_\delta(f^n), \\
		G_{i}^n &= \pdv{\abar^n_{ij}}{v_j} \beta_\delta(f^n)
		+\bbar_i^n (1-2\varepsilon_n f^n) \beta_\delta(f^n)
		+2\varepsilon_n \bbar_i^n B_\delta(f^n), \\
		G^n &=
		- \cbar^n \left[f^n(1-\varepsilon_n f^n)\beta_\delta'(f^n)
		- (1-2\varepsilon_n f^n)\beta_\delta(f^n)
		- 2\varepsilon_n B_\delta(f^n)\right] \\
		&\quad - \beta_\delta''(f^n) \abar^n_{ij} \pdv{f^n}{v_i} \pdv{f^n}{v_j}.
	\end{align*}

	Now, since $\beta_\delta(t) \leq \delta^{-1}$, $|1-2\varepsilon_n f^n| \leq 1$, and
	\[
		|\varepsilon_n B_\delta(f^n)|
		= \varepsilon_n \left( \frac{f^n}{\delta} - \frac{\log(\delta f^n + 1)}{\delta^2} \right)
		\leq \varepsilon_n\delta^{-1} f^n \leq \delta^{-1},
	\]
	it follows that the first two functions satisfy
	\[
		|G_{ij}^n| \leq \delta^{-1} |\abar^n_{ij}|
		\;\text{ and }\;
		|G_i^n| \leq \delta^{-1} \left|\pdv{\abar^n_{ij}}{v_j}\right| + 3\delta^{-1} |\bbar_i^n|.
	\]
	Thus, these sequences are bounded in $L^1_{loc} ((0,\infty) \times \Rxv)$.
	Moreover, since $t \beta_\delta'(t) =\frac{t}{(1+\delta t)^2} \leq \frac{1}{\delta}$, it follows that
	\[
		|G^n| \leq \frac{4}{\delta} |\cbar^n|
		- \beta_\delta''(f^n) \abar^n_{ij} \pdv{f^n}{v_i} \pdv{f^n}{v_j}
	\]
	and from Lemma \ref{ch2:lem-aij-v-bounded}, this sequence is bounded in
	$L^1_{loc}((0,\infty) \times \Rxv)$.
	Since $(\beta_\delta(f^n))_n$ is bounded in $L^\infty((0,\infty) \times \Rxv)$,
	Theorem \ref{ch2:bouchut-vel-avg} implies that for each $\delta > 0$, the velocity averages $\int_v \beta_\delta(f^n)\varphi$ are compact in $L^1_{loc}((0,T) \times K)$.

	\proofpart{II}{Velocity averages of $\beta_\delta(F^n)$.}
	The next part of the proof consists of passing from the compactness of the $\beta_\delta(f^n)$ averages to the compactness of the $\beta_\delta(F^n)$ averages.
	Notice that
	\[
		|\beta_\delta(F^n) - \beta_\delta(f^n)|
		= \varepsilon_n f^n \frac{f^n}{1+\delta f^n}\frac{1}{1+\delta F^n}
		\leq \delta^{-1} \varepsilon_n f^n,
	\]
	and thus we may estimate the difference of norms as
	\begin{equation}
		\label{ch2:eq:beta-delta-difference}
		\begin{split}
		\int_{tx} \left| \int_v \beta_\delta(F^n) \varphi - \int_v \beta_\delta(f^n) \varphi \right|
		&\leq \inttxv \delta^{-1} \varepsilon_n f^n |\varphi| \\
		&\leq \delta^{-1} \varepsilon_n \|\varphi\|_{L^\infty(\mathbb{R}^N)} \| f^n \|_{L^1((0,T) \times \Rxv)}.
		\end{split}
	\end{equation}
	Since the sequence $(f^n)_n$ is uniformly bounded in $L^1((0,\infty) \times \Rxv)$, the right-hand side of \eqref{ch2:eq:beta-delta-difference} vanishes as $n \to \infty$.
	This implies that the difference between the two sequences converges to zero in $L^1((0,T) \times K)$, and therefore $\int_v \beta_\delta(F^n) \varphi$ inherits the strong compactness of $\int_v \beta_\delta(f^n) \varphi$.

	\proofpart{III}{Velocity averages of $F^n$.}
	In this last step, we will use the fact that the $\delta > 0$ taken so far is arbitrary to obtain the compactness of the averages of $F^n$.
	We decompose, for $M>1$,
	\begin{equation*}
		| \beta_{\delta}(F^n) - F^n |
		= \frac{\delta(F^n)^2}{1+\delta F^n} \mathbbm{1}_{F^n \leq M}
		+ \frac{\delta(F^n)^2}{1+\delta F^n} \mathbbm{1}_{F^n > M}.
	\end{equation*}
	Using that $t \mapsto t/(1+\delta t)$ is an increasing function on the first term and that this map is bounded by $\delta^{-1}$ on the second one yields
	\begin{align*}
		| \beta_{\delta}(F^n) - F^n |
		&\leq \frac{M}{1+\delta M} \delta F^n \mathbbm{1}_{F^n \leq M}
			+ F^n \mathbbm{1}_{F^n > M} \\
		&\leq M\delta f^n
			+ f^n \mathbbm{1}_{f^n > M},
	\end{align*}
	where for the last inequality we have used that $F^n = f^n(1-\varepsilon_n f^n) \leq f^n$.
	Therefore, for every fixed $T>0$,
	\[
		\int_{tx} \left|
			\int_v \beta_{\delta}(F^n) \varphi
			- \int_v F^n \varphi \right|
		\leq M \delta \|\varphi\|_{L^\infty(\mathbb{R}^N)} \inttxv f^n
			+ \|\varphi\|_{L^\infty(\mathbb{R}^N)} \inttxv f^n \mathbbm{1}_{f^n > M}.
	\]

	Since $(f^n)_n$ is uniformly integrable, for every $\varepsilon > 0$ there exists an $M_{\varepsilon} > 0$ such that
	\[
		\sup_n \inttxv f^n \mathbbm{1}_{f^n > M_\varepsilon} \leq \varepsilon.
	\]
	This way, there exists a $C = C(T,\varphi) > 0$ such that for every $\delta \leq \varepsilon/M_\varepsilon$,
 	\[
		\left\|
			\int_v \beta_{\delta}(F^n) \varphi
			- \int_v F^n \varphi \right\|_{L^1((0,T) \times K)}
		\leq C(T,\varphi) \varepsilon,
	\]
	for every $n \in \mathbb{N}$.

	For each $m\in\mathbb{N}$, let $\delta_m > 0$ be chosen such that
	\begin{equation}
		\label{ch2:eq:beta-deltam-approx}
		\left\| \int_v \beta_{\delta_m} (F^n) \varphi
		- \int_v F^n\varphi
		\right\|_{L^1((0,T) \times K)}
		\leq \frac{1}{m},
	\end{equation}
	for every $n \in \mathbb{N}$.
	Let $\rho^n_m = \int_v \beta_{\delta_m}(F^n) \varphi$ denote the sequence of renormalizations and $\rho^n = \int_v F^n \varphi$ the sequence we wish to study.
	From Part II of the proof, we know that for each fixed $m$, the sequence $(\rho^n_m)_n$ is compact in $L^1((0,T) \times K)$.
	We will use a diagonal argument	to construct a single subsequence that converges for \emph{every}	$m$.

	We start with $m=1$.
	Since $(\rho^n_1)_n$ is compact, there exists a sequence of indices $\sigma_1(n)$ such that $\big( \rho^{\sigma_1(n)}_1 \big)_n$ converges.
	Next, since $\big( \rho^{\sigma_1(n)}_2 \big)_n$ is compact, we extract a subsequence $\sigma_2(n)$ of $\sigma_1(n)$ such that $( \rho^{\sigma_2(n)}_2 )_n$ converges.
	By iterating this process, we construct nested sequences of indices $\sigma_{k}$ such that $\big( \rho^{\sigma_k(n)}_\ell \big)_n$ converges for every $\ell \leq k$.

	Consider the diagonal sequence indices $n \mapsto \sigma_n(n)$.
	By definition, for $n$ large enough, this is a subsequence of $(\sigma_k(n))_n$ for any fixed $k$, and therefore $\big( \rho^{\sigma_n(n)}_k \big)_n$ converges for every $k \in \mathbb{N}$.

	We now show that $(\rho^{\sigma_n(n)})_n$ is a Cauchy sequence in $L^1((0,T) \times K)$.
	Let $\varepsilon > 0$ and choose $k$ large enough such that $2/k < \varepsilon/2$.
	Since $\big( \rho^{\sigma_n(n)}_k \big)_n$ is convergent for this fixed $k$, let $N(\varepsilon) \in \mathbb{N}$ be such that
	\[
		\left\| \rho^{\sigma_n(n)}_k - \rho^{\sigma_m(m)}_k \right\|_{L^1((0,T) \times K)}
		\leq \frac{\varepsilon}{2},
	\]
	for every $n,m \geq N(\varepsilon)$.

	By the triangle inequality and the approximation estimate \eqref{ch2:eq:beta-deltam-approx}, we have, for all $n, m \geq N(\varepsilon)$, that
	\begin{align*}
		\left\| \rho^{\sigma_n(n)} - \rho^{\sigma_m(m)} \right\|_{L^1((0,T) \times K)}
		&\leq \left\| \rho^{\sigma_n(n)}
			- \rho^{\sigma_n(n)}_k
		\right\|_{L^1((0,T) \times K)} \\
		&\quad  + \left\|
			\rho^{\sigma_n(n)}_k
			- \rho^{\sigma_m(m)}_k
		\right\|_{L^1((0,T) \times K)}\\
		&\quad  + \left\|
			\rho^{\sigma_m(m)}_k
			- \rho^{\sigma_m(m)}
		\right\|_{L^1((0,T) \times K)}\\
		&\leq\frac{2}{k}
		+ \frac{\varepsilon}{2}
		\\&\leq \varepsilon.
	\end{align*}
	Therefore, the sequence $(\rho^{\sigma_n(n)})_n$ is a Cauchy sequence in the Banach space\break $L^1((0,T) \times K)$ and therefore converges.
\end{proof}

\subsection{Almost everywhere compactness}
Once we have found some compactness in the variables $t$ and $x$ through velocity averaging, the aim now is to find compactness in the variable $v$.

The traditional way of doing this, widely used in the study of parabolic equations, is to search for an ellipticity estimate for the diffusion matrices $\abar^n_{ij}$ that is uniform in $n$, which together with Lemma \ref{ch2:lem-aij-v-bounded} would give us a bound on the derivatives of $f$.
Such an estimate, however, would fall apart if we had, for example, $\int_v f^n(1-\varepsilon_n f^n) = 0$ at some $(t,x)$, as this implies $\abar^n(t,x,v) = 0$ at that point.
Thus, this estimate depends primarily on the fact that we do not have a quantum vacuum at any $(t,x)$, which may be difficult to prove.

But, it turns out that such an elliptic estimate is not necessary, and we obtain sufficient compactness in $v$ using only a partial notion of ellipticity satisfied by $\abar$, namely Lemma 2 of \cite{sampaio-2024}, which is rewritten in our case as
\begin{lemma}
	\label{ch2:lemma-ellipticity}
	Fix $T, R > 0$.
	Let $(f^n)$ be a sequence of suitable weak solutions to \eqref{ch2:eq:lfd}.
	Let $(f^n_{m_n})_n$ be a sequence of functions such that $0 \leq f^n_{m_n} \leq 1$ for every $n \in \mathbb{N}$.
	Define $F^n_{m_n} \equiv f^n_{m_n} (1 - \varepsilon_n f^n_{m_n})$.
	Suppose there exists some $F \in L^\infty((0,\infty) \times \Rxv) \cap L^1_{loc}((0,\infty); L^1(\Rxv))$ such that
	\[
		\int F^n_{m_n} \varphi\, dv \to \int F \varphi\, dv
		\text{ in }
		L^1([0,T] \times B_R)
	\]
	for every $\varphi \in L^1(\mathbb{R}^N_v)$.
	Let $\abar^{n,m}$ be a sequence of matrices satisfying \eqref{ch2:ineq-uniform-ellip} and
	\[
		K_\alpha \equiv
		\left\{
			(t,x) \in [0,T] \times B_R : \int F\; dv > \alpha
		\right\}.
	\]

	Then, for every $\alpha, \varepsilon > 0$ there exists a measurable set $|E| < \varepsilon$ and a constant $C(\alpha, \varepsilon) > 0$ such that
	\[
		\abar^{n, m_n}_{ij}(t,x,v) \eta_i \eta_j \geq C(\alpha, \varepsilon) |\eta|^2,
	\]
	for $(x,t,v,\eta) \in (K_{\alpha} \cap E^c) \times B_R \times \mathbb{R}^N$ and $n \geq n_0(\varepsilon, \alpha)$.
\end{lemma}

This ellipticity, together with Lemma \ref{ch2:lem-aij-v-bounded}, gives us an estimate for the derivatives of $f$ in $v$, which implies compactness in the variable $v$.
Having already established compactness in $t$ and $x$ from the averaging lemmas in the last section, the next question is whether we can unify these two results into a single compactness in the three variables $t,x, ~\mbox{and}~v$.

The next proposition, which was also stated and proven in \cite{sampaio-2024}, shows us that this question is answered in the affirmative.
\begin{proposition}
	\label{ch2:prop-ae-compactness}
	Let $(\Phi^n)_n$ be a sequence of functions such that $\Phi^n \weakstarto \Phi$ in $L^\infty((0,\infty) \times \Rxv)$.
	Suppose that, for each $T, R > 0$, we have the properties

	\begin{enumerate}[label=\roman*.]
	\item \textbf{Quasi-bounded in $v$:} For every $\varepsilon, \alpha > 0$, there exists a $C = C(\varepsilon, \alpha)$ and a measurable set $E$, with $|E| \leq \varepsilon$ such that
	\[
		\iiint_{(K_\alpha \cap E^c) \times B_R} |\nabla_v \Phi^n|^2 \,dtdxdv \leq C(\varepsilon, \alpha),
	\]
	where
	$
		K_{\alpha} = \left\{ (t,x) \in [0,T] \times B_R : \int \Phi \, dv > \alpha \right\}.
	$

	\item \textbf{Velocity averages:} For every $\varphi \in C^\infty_c (\mathbb{R}^N_v)$, the sequence $(\int \Phi^n \varphi \,dv)_n$ is compact in $L^1([0,T] \times B_R)$.
	\end{enumerate}

	Thus, passing to a subsequence, we have $\Phi^n \to \Phi$ almost everywhere in $(0,\infty) \times \mathbb{R}^{2N}_{x,v}$.
\end{proposition}

Now we can finally gather all the above results to prove the result of this section.

\begin{proof}[Proof of Lemma \ref{ch2:lem:ae-comp}]
	Passing to a subsequence in $n$, we have $f^n_{m_n} \weakto f$ weakly in $L^1_{loc}((0,\infty); L^1(\Rxv))$ together with $F^n_{m_n} \weakstarto F$ and $\beta_{1}(f^n_{m_n}) \weakstarto b$ weakly-* in $L^\infty((0, \infty) \times \Rxv)$, where $F$ and $b$ are essentially bounded functions.

	Fix $T, R > 0$.
	From Lemma \ref{ch2:lem-aij-v-bounded}, there exists a constant $C > 0$ such that
	\[
		\iiint_{(0,T) \times B_R \times B_R} \abar^{n,m_n}_{ij} 	\pdv{\beta_1(f^n_{m_n})}{v_i} \pdv{\beta_1(f^n_{m_n})}{v_j} \, dtdxdv
		\leq C.
	\]
	Combining this with Lemma \ref{ch2:lem:velocity-averages-converge} and the ellipticity estimate in Lemma \ref{ch2:lemma-ellipticity}, we obtain that for every $\alpha, \varepsilon > 0$, there exists a measurable set $E$ with $|E| \leq \varepsilon$ such that
	\begin{equation}
		\label{ch2:eq:beta_1-derivative-bound}
		\iiint_{(K_\alpha \cap E^c) \times B_R}
		\left|\nabla_v \beta_1(f^n_{m_n})\right|^2 \, dtdxdv
		\leq C(\varepsilon,\alpha) \, C.
	\end{equation}
	Here, the set $K_\alpha$ is defined by
	\[
	    K_\alpha \equiv \left\{ (t,x) \in [0,T] \times B_R : \int F \, dv > \alpha \right\}.
	\]

	Proposition \ref{ch2:prop-ae-compactness} is not directly applicable here because the set $K_\alpha$ defined in terms of $\int F \; dv$ differs from the set required by the proposition, which relies on $\int b \; dv$.

	To overcome this, first observe that
	\begin{equation}
	\begin{split}
		| \beta_{1}(f^n_{m_n}) - \beta_1(F^n_{m_n}) |
		= \frac{\varepsilon_n (f^n_{m_n})^2}{(1 + f^n_{m_n})(1 + F^n_{m_n})}
		\leq \varepsilon_n f^n_{m_n},
	\end{split}
	\end{equation}
	where we used the inequality $x/(1+x) \leq 1$ for every $x \geq 0$.
	Consequently, for any nonnegative test function $\varphi \in C^\infty((0,\infty) \times \Rxv)$, we have
	\[
		\left|
			\int_{txv} \beta_{1}(f^n_{m_n}) \varphi
			- \int_{txv} \beta_{1}(F^n_{m_n}) \varphi
		\right|
		\leq \varepsilon_n \int_{txv} f^n_{m_n} \varphi.
	\]
	Using the fact that $\beta_1(t) \leq t$ for every $t \geq 0$, this implies
	\[
		\int_{txv} \beta_{1}(f^n_{m_n}) \varphi
		\leq \int_{txv} F^n_{m_n} \varphi
		+ \varepsilon_n \int_{txv} f^n_{m_n} \varphi.
	\]
	Passing to the limit $n \to \infty$, we obtain
	\[
		\int_{txv} b \varphi \leq \int_{txv} F \varphi
	\]
	and since the smooth function $\varphi$ is arbitrary, we conclude that $b \leq F$ almost everywhere in $(0,\infty) \times \Rxv$, which implies that if we define
	\[
		K'_\alpha \equiv
		\left\{
			(t,x) \in [0,T] \times B_R : \int b \, dv > \alpha
		\right\},
	\]
	then $K'_\alpha	\subset	K_\alpha$.
	Therefore, the estimate \eqref{ch2:eq:beta_1-derivative-bound} implies that
	\[
		\iiint_{(K'_\alpha \cap E^c) \times B_R}
		\left|\nabla_v \beta_1(f^n_{m_n})\right|^2 \, dtdxdv
		\leq C(\varepsilon,\alpha).
	\]

	By Lemma \ref{ch2:lem:velocity-averages-converge}, the velocity averages $\int \beta_{1}(f^n_{m_n}) \varphi \; dv$ are compact in $L^1_{loc}((0,\infty) \times \mathbb{R}^N_x)$.
	Proposition \ref{ch2:prop-ae-compactness} then implies that, passing to a subsequence, $(\beta_1(f^n_{m_n}))_n$ converges almost everywhere in $(0,\infty) \times \Rxv$.
	This, in turn, implies that $(f^n_{m_n})_n$ converges almost everywhere in the same space.

	Finally, since the sequence is weakly compact in $L^1_{loc}((0,\infty); L^1(\Rxv))$, we infer strong compactness in this space using Scheffé's lemma.
\end{proof}

\section{Proof of Theorem \ref{ch2:thm:main-theorem}}
\label{ch2:sec:thm-1-proof}
We now have all the necessary ingredients for the proof of Theorem \ref{ch2:thm:main-theorem}, which we will divide into two sections.

In the first part, we will apply the compactness results we have seen so far to show that the solutions of the Landau--Fermi--Dirac equation converge to renormalized solutions of the Landau equation, as in Villani's definition.

In the second part of the proof, we will show the convergence of the conservation laws and the dissipation of entropy, showing that the entropy inequality is still valid for solutions in the semi-classical limit.

\subsection{Convergence to renormalized solutions}
For each $n \in \mathbb{N}$, let $(f^n_m)_m$ be an approximation scheme converging to $f^n$, with coefficients $\abar^{n,m}$ and $\bbar^{n,m}$.

Let $B_n = (0,n) \times B_x(0,n) \times B_v(0,n)$ be the $(t,x,v)$ ball with radius $n$.
From the convergence $f^n_m \to f^n$ in $L^1_{loc}((0,\infty); L^1(\Rxv))$, as well as the convergences of the coefficients from Definition \ref{ch2:def:suitable-weak-solution}, for every $n \in \mathbb{N}$ there exists some $M_0(n) \in \mathbb{N}$ such that, for every $m \geq M_0(n)$,
\begin{equation}
	\label{ch2:eq:fn-fnm-l1-norm}
	\| f^n - f^n_m \|_{L^1((0,n); L^1(\Rxv))} \leq \frac{1}{n}
\end{equation}
as well as
\begin{equation}
	\label{ch2:ineq:coeff-conv-1}
	\begin{split}
	\| \abar^{n,m}_{ij} - \abar^n_{ij} \|_{L^1(B_n)} &\leq \frac{1}{n}, \\
	\left\| \pdv{\abar^{n,m}_{ij}}{v_j} - \pdv{\abar^n_{ij}}{v_j} \right\|_{L^1(B_n)} &\leq \frac{1}{n}, \\
	\| \bbar^{n,m}_i - \bbar^n_i \|_{L^1(B_n)} &\leq \frac{1}{n}, \\
	\| \cbar^{n,m} - \cbar^n \|_{L^1(B_n)} &\leq \frac{1}{n}.
	\end{split}
\end{equation}

This way, chaining Lemmas \ref{ch2:lem:diag-weak-convergence} and \ref{ch2:lem:ae-comp}, we obtain a diagonal sequence $(f^n_{m_n})_n$ that is strongly compact in $L^1_{loc}((0,\infty); L^1(\Rxv))$, and we can then pass to a convergent subsequence.

For conciseness, we will write $f^n_n$ instead of $f^n_{m_n}$, as well as $\abar^{n,n}$ instead of $\abar^{n,m_n}$ for the other coefficients and so on.
Passing to a further subsequence, we have that
\begin{equation}
	\label{ch2:conv:fnn-convergence}
	f^n_n \to f \text{ in } L^1_{loc}((0,\infty); L^1(\Rxv))
	\text{ and a.e. in } (0,\infty) \times \Rxv,
\end{equation}
and using \eqref{ch2:eq:fn-fnm-l1-norm}, we also have, passing to subsequences if necessary,
\begin{equation}
	\label{ch2:conv:fn-convergence}
	f^n \to f \text{ in } L^1_{loc}((0,\infty); L^1(\Rxv))
	\text{ and a.e. in } (0,\infty) \times \Rxv.
\end{equation}

Since $f^n_n \in L^2_{loc}((0,T) \times \mathbb{R}^N_x; H^1_{loc}(\mathbb{R}^N_v))$, Proposition \ref{ch2:prop:approxim-implies-renorm} ensures that $f^n_n$ satisfies the renormalized formulation, that is, for every $\delta > 0$, taking $\beta_{\delta}(t) = t/(1+\delta t)$, we have
\begin{equation}
\label{ch2:eq:renormalized-fnn-to-the-limit}
\begin{split}
	&\pdv{\beta_\delta(f^n_n)}{t}+ v\cdot\nabla_x \beta_\delta(f^n_n)\\
	& = \pdv{}{v_i}
	\left\{
		\abar^{n,n}_{ij} \pdv{\beta_\delta(f^n_n)}{v_j}
		-\bbar^{n,n}_i (1 - 2\varepsilon_n f^n_n) \beta_\delta(f^n_n)
		- 2 \varepsilon_n \bbar^{n,n}_i B_\delta(f^n_n)
	\right\}\\
	&\quad  -\beta_\delta''(f^n_n) \abar^{n,n}_{ij} \pdv{f^n_n}{v_i} \pdv{f^n_n}{v_j}\\
	&\quad  -\cbar^{n,n} \left[
		f^n_n (1 - \varepsilon_n f^n_n)\beta_\delta'(f^n_n)
		- (1 - 2\varepsilon_n f^n_n)\beta_\delta(f^n_n)
		- 2\varepsilon_n B_\delta(f^n_n)
	\right]
\end{split}
\end{equation}
in $\mathcal{D}'((0,\infty) \times \Rxv)$.

From \eqref{ch2:conv:fnn-convergence}, we have $\beta_\delta(f^n_n) \to \beta_\delta (f)$ a.e., and thus, since $|\beta_\delta(f^n_n)| \leq \frac{1}{\delta}$, it follows that
\[
	\pdv{\beta_\delta(f^n_n)}{t} + v \cdot \nabla_x \beta_\delta(f^n_n)
	\to
	\pdv{\beta_\delta(f)}{t} + v \cdot \nabla_x \beta_\delta(f)
\]
in $\mathcal{D}'((0,\infty) \times \Rxv)$.

Next, let us prove the convergence of the coefficients $\abar^{n,n}_{ij}$, $\pdv{\abar^{n,n}_{ij}}{v_j}$, $\bbar^{n,n}_i$, and $\cbar^{n,n}$.
Consider $\abar^n_{ij} = a_{ij} *_v (f^n(1-\varepsilon_n f^n))$.
We need a strong convergence result for $f^n(1-\varepsilon_n f^n)$.
Using \eqref{ch2:conv:fn-convergence}, we obtain that $f^n (1 - \varepsilon_n f^n) \to f$ almost everywhere in $(0,\infty) \times \Rxv$.
Since for every $n$ we have $\left|f^n (1-\varepsilon_n f^n)\right| \leq f^n$, an application of the general dominated convergence theorem implies that
\begin{equation}
\label{ch2:conv:fn-1-fn-convergence}
	f^n (1 - \varepsilon_n f^n)\to f
	\quad\text{in}\quad
	L^1_{loc}((0,\infty); L^1(\Rxv)).
\end{equation}

Recall that the convolution of a function $g$ with a signed measure $\mu$ satisfies the variant of Young's inequality
\[
	\| \mu * g \|_{L^1(\mathbb{R}^N)} \leq \| g \|_{L^1(\mathbb{R}^N)} |\mu|(\mathbb{R}^N).
\]
Therefore, if $g^n \to g$ in $L^1$, then $\mu * g^n \to \mu * g$ in $L^1$.
Hence, the regularity hypothesis for $a$ in \eqref{ch2:property:a-regularity} together with the convergences \eqref{ch2:conv:fn-convergence} and \eqref{ch2:conv:fn-1-fn-convergence} yields that $\abar^n_{ij} \to \abar_{ij}$, $\pdv{\abar^n_{ij}}{v_j} \to \pdv{\abar_{ij}}{v_j}$, $\bbar^n_i \to \bbar_i$, and $\cbar^n \to \cbar$ in $L^1_{loc}((0,\infty) \times \Rxv)$.
Then, using \eqref{ch2:ineq:coeff-conv-1}, we have that
\[
	\abar^{n,n}_{ij} \to \abar_{ij}, \;
	\pdv{\abar^{n,n}_{ij}}{v_j} \to \pdv{\abar_{ij}}{v_j}, \;
	\bbar^{n,n}_i \to \bbar_i, \;
	\cbar^{n,n} \to \cbar
\]
in $L^1_{loc}((0,\infty) \times \Rxv)$ and, by passing to a further subsequence, we also have these convergences almost everywhere in $(0,\infty) \times \Rxv$.

Let us pass each of the terms on the right-hand side of \eqref{ch2:eq:renormalized-fnn-to-the-limit} to the limit in the sense of distributions.
We treat the diffusion part first.
Following Remark \ref{ch2:rem:product-rule-interp}, consider
\[
	\abar^{n,n}_{ij} \pdv{\beta_\delta(f^n_n)}{v_j}
	= \pdv{}{v_j} \left[ \abar^{n,n}_{ij} \beta_\delta(f^n_n) \right]
		- \pdv{\abar^{n,n}_{ij}}{v_j} \beta_\delta(f^n_n).
\]
Since $|\beta_\delta(f^n_n)| \leq \delta^{-1}$, we have from dominated convergence that $\beta_\delta(f^n_n) \weakstarto \beta_\delta(f)$ weakly-* in $L^\infty((0, \infty) \times \Rxv)$, and by weak-strong convergence,
\[
	\abar^{n,n}_{ij} \beta_\delta(f^n_n) \to \abar_{ij} \beta_\delta(f)
	\quad \text{and} \quad
	\pdv{\abar^{n,n}_{ij}}{v_j} \beta_\delta(f^n_n) \to \pdv{\abar_{ij}}{v_j} \beta_\delta(f)
\]
in $\mathcal{D}'((0,\infty) \times \Rxv)$, which implies
\[
	\abar^{n,n}_{ij} \pdv{\beta_\delta(f^n_n)}{v_j}
	\to \pdv{}{v_i} \left[ \abar_{ij} \beta_\delta(f) \right]
		- \pdv{\abar_{ij}}{v_j} \beta_\delta(f)
\]
in $\mathcal{D}'((0,\infty) \times \Rxv)$.
We denote by $\abar_{ij} \pdv{\beta_\delta(f)}{v_j}$ the right-hand side of the above limit.

Next, since $0 \leq f^n_n \leq \varepsilon_n^{-1}$, we have that $|1 - 2\varepsilon_n f^n_n| \leq 3$, and dominated convergence implies that
\[
	(1 - 2\varepsilon_n f^n_n) \beta_\delta(f^n_n) \weakstarto \beta_\delta(f) \text{ in } L^\infty((0,\infty) \times \Rxv).
\]
Furthermore, using that $|B_\delta(t)| \leq t/\delta$, for every $t \geq 0$, we have that
\[
	\varepsilon_n B_\delta(f^n_n) \weakstarto 0 \text{ in } L^\infty((0,\infty) \times \Rxv),
\]
again by dominated convergence.
Hence, as before, we have by weak-strong convergence that
\[
	\pdv{}{v_i} \left[
		\bbar^{n,n}_i (1 - 2\varepsilon_n f^n_n) \beta_\delta(f^n_n)
		+ 2\varepsilon_n \bbar^{n,n}_i B_\delta(f^n_n)
	\right]
	\to
	\pdv{}{v_i} \left[	\bbar_i \beta_\delta(f) \right]
\]
in $\mathcal{D}'((0, \infty) \times \Rxv)$.

Let us now pass to the limit the third term on the left-hand side of \eqref{ch2:eq:renormalized-fnn-to-the-limit}.
Since $| f^n_n (1 - \varepsilon_n f^n_n) | \leq f^n_n$ and $t \beta'_\delta(t) = \frac{t}{(1+\delta t)^2} \leq \frac{1}{4 \delta}$ for every $t \geq 0$, we have that
\[
	f^n_n (1 - \varepsilon_n f^n_n) \beta'_\delta(f^n_n) \weakstarto f \beta'_\delta(f) \text{ in } L^\infty((0,\infty) \times \Rxv).
\]
Thus, from weak-strong convergence,
\begin{equation*}
\begin{split}
	&\cbar^{n,n} \big[
		f^n_n (1 - \varepsilon_n f^n_n) \beta'_\delta(f^n_n)
		- (1 - 2\varepsilon_n f^n_n) \beta_\delta(f^n_n)
- 2\varepsilon_n B_\delta(f^n_n)
	\big] \\
	& \to
	\cbar \left[ f\beta'_\delta(f) - \beta_\delta(f) \right]
\end{split}
\end{equation*}
in $\mathcal{D}'((0, \infty) \times \Rxv)$.

It remains  to prove the convergence of the quadratic term
\begin{equation}
	\label{ch2:eq:thm1-abar-fn-quadratic}
	-\beta_\delta''(f^n_n) \abar^{n,n}_{ij} \pdv{f^n_n}{v_i} \pdv{f^n_n}{v_j}
\end{equation}
in the sense of distributions.
Defining the function
\[
	\gamma_\delta(t) = - \frac{2\sqrt{2}}{\sqrt{\delta(1+\delta t)}},
\]
we notice that this implies $(\gamma'_\delta(t))^2 = \frac{2\delta}{(1 + \delta t)^3} = - \beta_\delta''(t)$.
Using \eqref{ch2:eq:abar-m-elliptic}, we can bound \eqref{ch2:eq:thm1-abar-fn-quadratic} from below as
\begin{equation*}
\begin{split}
	-\beta_\delta''(f^n_n) \abar^{n,n}_{ij} \pdv{f^n_n}{v_i} \pdv{f^n_n}{v_j}
	&= \abar^{n,n}_{ij} \pdv{\gamma_\delta(f^n_n)}{v_i} \pdv{\gamma_\delta(f^n_n)}{v_j} \\
	&\geq \big[
			a_{m_n} *_v (f^n_n (1-\varepsilon_n f^n_n))
			\big]
			\pdv{\gamma_\delta(f^n_n)}{v_i} \pdv{\gamma_\delta(f^n_n)}{v_j} \\
	&= \int_{v_*} a_{m_n}(v-v_*) f^n_n(v_*) (1-\varepsilon_n f^n_n(v_*))
			\pdv{\gamma_\delta(f^n_n)}{v_i} \pdv{\gamma_\delta(f^n_n)}{v_j} \\
	&= \int_{v_*}
				\left|
					\sqrt{a_{m_n}(v-v_*) } \sqrt{f^n_n(v_*) (1-\varepsilon_n f^n_n(v_*))} \nabla_v \gamma_\delta(f^n_n)
				\right|^2,
\end{split}
\end{equation*}
where $\sqrt{a_{m_n}(v-v_*)}$ denotes the matrix square root of $a_{m_n}(v-v_*)$.
For conciseness of notation, let us write $a_n$ instead of $a_{m_n}$.
For every $n \in \mathbb{N}$, we define the nonnegative Radon measure $\nu^n$ on $(0,\infty) \times \Rxv$ by
\begin{equation}
	\label{ch2:eq:thm-1-proof-mu-n}
	\begin{split}
\nu^n \equiv&
	-\beta_\delta''(f^n_n) \abar^{n,n}_{ij} \pdv{f^n_n}{v_i} \pdv{f^n_n}{v_j} \\
	& - \int_{v_*} \left|
		\sqrt{a_{n}(v-v_*) } \sqrt{f^n_n(v_*) (1-\varepsilon_n f^n_n(v_*))} \nabla_v \gamma_\delta(f^n_n)
	\right|^2.
	\end{split}
\end{equation}
Since the second term is non-positive, the expression of $\nu^n$ is bounded from above by its first term,
\[
	0 \leq \nu^n
	\leq -\beta_\delta''(f^n_n) \abar^{n,n}_{ij} \pdv{f^n_n}{v_i} \pdv{f^n_n}{v_j}.
\]
By Lemma \ref{ch2:lem-aij-v-bounded}, the right-hand side is uniformly bounded in $L^1_{loc}((0,\infty) \times \Rxv)$.
Therefore, up to extraction of a subsequence, there exists a nonnegative Radon measure $\nu$ such that
\begin{equation}
	\label{ch2:eq:nu-convergence}
	\nu^n \weakto \nu
	\quad \text{weakly in }
	\mathcal{M}_{loc} ((0,\infty) \times \Rxv)
\end{equation}
(see, for example, Theorem 1.41 in \cite{evans_2015}).

We rewrite the term inside the square of \eqref{ch2:eq:thm-1-proof-mu-n} as
\begin{equation}
	\label{ch2:eq:fnn-quadratic}
	\begin{split}
	&\sqrt{a_n(v-v_*) }\sqrt{f^n_n(v_*) (1-\varepsilon_n f^n_n(v_*))} \nabla_v \gamma_\delta(f^n_n) \\
	&=\Div_v \left[
		\sqrt{a_n(v-v_*)} \sqrt{f^n_n(v_*) (1-\varepsilon_n f^n_n(v_*))} \gamma_\delta(f^n_n)
		\right]\\
	&\quad - \Div_v \left[ \sqrt{a_n(v-v_*)} \right] \sqrt{f^n_n(v_*) (1-\varepsilon_n f^n_n(v_*))} \gamma_\delta(f^n_n).
	\end{split}
\end{equation}
The matrices $a_n$ have a fixed form, given by \eqref{ch2:eq:abar-m-elliptic}, and therefore by the uniqueness of the square root of positive semi-definite matrices, we have
\begin{equation}
	\label{ch2:eq:sqrt-a_m}
	\sqrt{a_n(z)}
	= \sqrt{\Gamma_n(|z|)} \left(I - \frac{z \otimes z}{|z|^2}\right).
\end{equation}
This implies, thanks to the convergence of $\Gamma_n(|\cdot|)$, that $\sqrt{a_n(z)} \to \sqrt{a(z)}$ in $L^{2r}_{loc}(\mathbb{R}^N)$.
Moreover, since
\[
	\sqrt{f^n_n(v_*) (1-\varepsilon_n f^n_n(v_*))} \to \sqrt{f_*}
	\text{ in }
	L^2_{loc}((0,\infty) \times \mathbb{R}^{3N}_{xvv_*})
\]
and $\gamma_\delta(f^n_n) \to \gamma_\delta(f)$ in $L^1_{loc}((0,\infty) \times \mathbb{R}^{3N}_{xvv_*})$, we can deduce that
\[
	\sqrt{a_n(v-v_*)} \sqrt{f^n_n(v_*) (1-\varepsilon_n f^n_n(v_*))} \gamma_\delta(f^n_n)
	\to \sqrt{a(v-v_*)} \sqrt{f_*} \gamma_\delta(f)
\]
in $L^1_{loc}((0,\infty) \times \mathbb{R}^{3N}_{xvv_*})$.

Next, notice that the explicit formula \eqref{ch2:eq:sqrt-a_m} allows us to calculate the divergence of the matrix $\sqrt{a}$, giving
\[
	\Div_z \left[ \sqrt{a_n(z)} \right]
	= - (N-1) \sqrt{\Gamma_n(|z|)} \frac{z}{|z|^2}.
\]
In view of the convergence of $\Gamma_n(|\cdot|)$ in the hypothesis of the theorem, this converges to $\Div_z (\sqrt{a(z)})$ in $L^1_{loc}(\mathbb{R}^N)$.
Therefore, we have that
\begin{equation}
	\label{ch2:eq:thm-1-convergence-1}
	\begin{split}
	&\sqrt{a_n(v-v_*) }\sqrt{f^n_n(v_*) (1-\varepsilon_n f^n_n(v_*))} \nabla_v \gamma_\delta(f^n_n)\\
	&\to
	\Div_v \left[ \sqrt{a(v-v_*)} \sqrt{f_*} \gamma_\delta(f) \right]
	- \Div_v \left[ \sqrt{a(v-v_*)} \right] \sqrt{f_*} \gamma_\delta(f)
	\end{split}
\end{equation}
in $\mathcal{D}'((0,\infty) \times \Rxv)$.
We denote the limit distribution as $\sqrt{a(v-v_*)} \sqrt{f_*} \nabla_v \gamma_\delta(f)$.

Lemma \ref{ch2:lem-aij-v-bounded} gives us that \eqref{ch2:eq:thm1-abar-fn-quadratic} is bounded in $L^1_{loc}((0,\infty) \times \Rxv)$, which then implies that the right-hand side of \eqref{ch2:eq:fnn-quadratic} is uniformly bounded in $L^2_{loc}((0,\infty) \times \Rxv; L^2(\mathbb{R}^N_{v_*}))$.
Therefore, the convergence \eqref{ch2:eq:thm-1-convergence-1} actually holds weakly in this space and the limit distribution $\sqrt{a(v-v_*)} \sqrt{f_*} \nabla_v \gamma_\delta(f)$ can be represented by a function belonging to $L^2_{loc}((0,\infty) \times \Rxv; L^2(\mathbb{R}^N_{v_*}))$.

We are interested in the convergence of the integrated quantity
\[
    I_n \equiv \int_{v_*} \left| \sqrt{a_n(v-v_*) } \sqrt{f^n_n(v_*) (1-\varepsilon_n f^n_n(v_*))} \nabla_v \gamma_\delta(f^n_n) \right|^2.
\]
From Lemma \ref{ch2:lem-aij-v-bounded}, the sequence $I_n$ is uniformly bounded in $L^1_{loc}((0,\infty) \times \Rxv)$.
Thus, up to a subsequence, there exists a nonnegative Radon measure $\nu'$ such that
\begin{equation}
	\label{ch2:eq:In-measure-conv}
    I_n \rightharpoonup \nu' \quad \text{weakly in } \mathcal{M}_{loc}((0,\infty) \times \Rxv).
\end{equation}
It remains to relate $\nu'$ to the weak limit $\sqrt{a(v-v_*)} \sqrt{f_*} \nabla_v \gamma_\delta(f)$.
From the lower semicontinuity of the $L^2$ norm with respect to the weak topology, we have that, for every compact set $K \subset (0,\infty) \times \Rxv$,
\[
	\liminf_{n\to\infty} \int_K I_n \, dtdxdv
	\geq \int_K \int_{\mathbb{R}^N}
		\left|
			\sqrt{a(v-v_*)} \sqrt{f_*} \nabla_v \gamma_\delta(f)
		\right|^2\, dv_* \, dtdxdv.
\]
By the Portmanteau theorem for the weak convergence \eqref{ch2:eq:In-measure-conv} (see, for example, Theorem 1.40 of \cite{evans_2015}), it follows that
\[
	\limsup_{n\to\infty} \int_K I_n \, dtdxdv \leq \nu(K)
\]
and since the compact set $K$ is arbitrary, we have, in the sense of measures, that
\[
    \nu' - \int_{v_*} |\sqrt{a(v-v_*)} \sqrt{f_*} \nabla_v \gamma_\delta(f)|^2 \geq 0.
\]
We define the defect measure $\nu''$ as the expression on the left-hand side.
Substituting back into \eqref{ch2:eq:In-measure-conv}, we obtain that
\begin{equation*}
\begin{split}
   & \int_{v_*}
    \Big|
        \sqrt{a_n(v-v_*)} \sqrt{f^n_n(v_*) (1-\varepsilon_n f^n_n(v_*))}\nabla_v \gamma_\delta(f^n_n)
    \Big|^2 \\
    &\weakto
    \int_{v_*}
    \Big|
        \sqrt{a(v-v_*)} \sqrt{f_*} \nabla_v \gamma_\delta(f)
    \Big|^2
    + \nu''
\end{split}
\end{equation*}
weakly in $\mathcal{M}_{loc}((0,\infty) \times \Rxv)$.
Combining this with the convergence \eqref{ch2:eq:nu-convergence}, it follows that there exists some nonnegative defect measure $\mu \in \mathcal{M}_{loc}((0,\infty) \times \Rxv)$ such that
\[
	-\beta_\delta''(f^n_n) \abar^{n,n}_{ij} \pdv{f^n_n}{v_i} \pdv{f^n_n}{v_j}
	\to
	\int_{v_*}
	\Big|
		\sqrt{a(v-v_*)} \sqrt{f_*} \nabla_v \gamma_\delta(f)
	\Big|^2
	+ \mu
\]
in $\mathcal{D}'((0, \infty) \times \mathbb{R}^{3N}_{xvv_*})$.

\subsection{Conservation laws and entropy inequality}
So far, we have used the approximation schemes $(f^n_n)_n$ to show that $f$ satisfies the renormalized Landau equation with defect measure.
In this second part of the proof, we will use the fact that the $(f^n)_n$ are weak solutions of the LFD equation to conclude that $f$ satisfies the conservation laws and inequalities of Definition \ref{ch2:def-renormalized-landau}.

Recall the convergences \eqref{ch2:conv:fn-convergence}.
First, from the $L^1$ convergence, we can infer that
\begin{equation}
	\label{ch2:eq:limit-f-conservation-mass}
	\intxv f(t)
	= \lim_{n \to \infty} \intxv f^n(t)
	= \lim_{n \to \infty} \intxv f^n_0
	= \intxv f_0,
\end{equation}
which implies the conservation of mass.
Next, we have from Fatou's lemma, together with the a.e. convergence, that
\[
	\intxv f(t) \psi
	\leq \liminf_{n \to \infty} \intxv f^n(t) \psi
	\leq \liminf_{n \to \infty} \intxv f^n_0 \psi
	= \intxv f_0 \psi,
\]
for $\psi = |v|^2$ or $\psi = |x-tv|^2$, which implies the inequalities of kinetic energy and moment of inertia, respectively.
Finally, by interpolation, we find that
\[
	\intxv f(t) v = \intxv f_0 v,
\]
which corresponds to the conservation of linear momentum.

\subsection{Entropy inequality}

Finally, let us show that the entropy inequality holds for the limit solution $f$.
Since $f^n$ is a weak solution of LFD with quantum parameter $\varepsilon_n$ and initial data $f^n_0$, it obeys the quantum entropy inequality,
\begin{equation}
	\label{ch2:eq:entropy-inequality-proof-1}
	\begin{split}
&	\frac{1}{\varepsilon_n} \intxv \varepsilon_n f^n \log(\varepsilon_nf^n) + (1-\varepsilon_n f^n) \log(1-\varepsilon_n f^n)
	+ \frac{1}{2} \int_0^t \intxv d^n
	\\
	& \leq \frac{1}{\varepsilon_n} \intxv \varepsilon_n f^n_0 \log(\varepsilon_n f^n_0) + (1-\varepsilon_n f^n_0) \log(1-\varepsilon_n f^n_0).
	\end{split}
\end{equation}
Here, $d^n$ is the entropy dissipation \eqref{ch2:dissipation-def} with quantum parameter $\varepsilon_n$, where we use the interpretation from Remark \ref{ch2:rem:entropy-dissipation-definition}.
Using the conservation of mass \eqref{ch2:eq:limit-f-conservation-mass}, we rewrite inequality \eqref{ch2:eq:entropy-inequality-proof-1} as
\begin{equation*}
\begin{split}
	&\intxv f^n \log f^n + \frac{1-\varepsilon_n f^n}{\varepsilon_n} \log(1-\varepsilon_n f^n)
	+ \frac{1}{2} \int_0^t \intxv d^n\\
	& \leq
	\intxv f^n_0 \log f^n_0 + \frac{1-\varepsilon_n f^n_0}{\varepsilon_n} \log(1-\varepsilon_n f^n_0).
\end{split}
\end{equation*}
Since for every $0 \leq x \leq \varepsilon^{-1}$ one has
\[
	-x \leq \frac{1-\varepsilon x}{\varepsilon} \log(1-\varepsilon x) \leq -x + \varepsilon x^2,
\]
the above inequality becomes
\begin{equation}
	\label{ch2:eq:entropy-inequality-proof-2}
	\intxv f^n \log f^n
	+ \frac{1}{2} \int_0^t \intxv d^n
	\leq
	\intxv f^n_0 \log f^n_0 + \varepsilon_n \intxv (f^n_0)^2,
\end{equation}
where we have used once again the conservation \eqref{ch2:eq:limit-f-conservation-mass}.

Since $\varepsilon_n (f^n_0)^2 \leq f^n_0$, we have from \eqref{ch2:eq:initial-data-convergence} and dominated convergence that the right-hand side converges to
\[
	\intxv f_0 \log f_0.
\]
Using the almost everywhere convergence in \eqref{ch2:conv:fn-convergence} and passing both sides of \eqref{ch2:eq:entropy-inequality-proof-2} to the $\liminf$, we have from Fatou's lemma that
\begin{equation}
	\label{ch2:eq:entropy-inequality-proof-3}
	\intxv f \log f + \frac{1}{2} \liminf_{n \to \infty} \int_0^t \intxv d^n
	\leq \intxv f_0 \log f_0.
\end{equation}
So, all that remains is to pass the entropy dissipation $d^n$ to the limit.

From Remark \ref{ch2:rem:entropy-dissipation-definition}, we recall that, since we do not suppose any regularity of $f$ in $v$, the term $\int_0^t \intxv d^n$ should instead be viewed as a notation for the square of the $L^2((0,t) \times \mathbb{R}^{3N}_{xvv_*})$ norm of $D^n$, where $D^n$ is the function
\begin{equation*}
\begin{split}
	D^n &= \Div_v \left(
			\sqrt{a(v-v_*)}
			\sqrt{f^n_* (1-\varepsilon_n f^n_*)} \frac{2}{\sqrt{\varepsilon_n}}
			\arcsin\sqrt{\varepsilon_n f^n}
		\right)\\
	&\quad - \Div_v \left( \sqrt{a(v-v_*)} \right) \sqrt{f^n_* (1-\varepsilon_n f^n_*)} \frac{2}{\sqrt{\varepsilon_n}} \arcsin \sqrt{\varepsilon_n f^n}\\
	&\quad - \Div_{v_*} \left(
			\sqrt{a(v-v_*)}
			\sqrt{f^n (1-\varepsilon_n f^n)} \frac{2}{\sqrt{\varepsilon_n}}
			\arcsin\sqrt{\varepsilon_n f^n_*}
		\right)\\
	&\quad + \Div_{v_*} \left( \sqrt{a(v-v_*)} \right) \sqrt{f^n(1-\varepsilon_n f^n)} \frac{2}{\sqrt{\varepsilon_n}} \arcsin \sqrt{\varepsilon_n f^n_*}.
\end{split}
\end{equation*}

From the almost everywhere convergence in \eqref{ch2:conv:fn-convergence}, we can deduce that
\begin{equation}
	\label{ch2:eq:entropy-dissipation-conv-1}
	\sqrt{a(v-v_*)} \sqrt{f^n_* (1-\varepsilon_n f^n_*)}
	\frac{\arcsin \sqrt{\varepsilon_n f^n}}{\sqrt{\varepsilon_n f^n}} \sqrt{f^n}
	\to
	\sqrt{a(v-v_*)} \sqrt{f f_*}
\end{equation}
for almost every $(t,x,v,v_*) \in (0,\infty) \times \mathbb{R}^{3N}_{xvv_*}$.
We also have that
\[
	\left|
		\sqrt{a(v-v_*)} \sqrt{f^n_* (1-\varepsilon_n f^n_*)}
		\frac{\arcsin \sqrt{\varepsilon_n f^n}}{\sqrt{\varepsilon_n f^n}} \sqrt{f^n}
	\right|
	\leq \frac{\pi}{2} \left| \sqrt{a(v-v_*)} \sqrt{f^n_*} \sqrt{f^n} \right|.
\]

The convergence \eqref{ch2:conv:fn-convergence} implies there exists a function $g$ in $L^1_{loc}((0,\infty), L^1(\Rxv))$ such that, passing to a subsequence in $n$, we have $|f^n| \leq g$ uniformly in $n$.
This in turn implies
\[
	\left| \sqrt{a(v-v_*)} \sqrt{f^n_*} \sqrt{f^n} \right|
	\leq \left| \sqrt{a(v-v_*)} \sqrt{g_*} \sqrt{g} \right|.
\]

Let us show that this function is in $L^1_{loc}((0,\infty) \times \mathbb{R}^{3N}_{xvv_*})$.
Indeed, from the uniqueness of the square root of positive matrices,
\begin{equation}
	\label{ch2:eq:sqrt-a}
	\sqrt{a(z)}
	= \sqrt{\Gamma(|z|)} \left(I - \frac{z \otimes z}{|z|^2}\right).
\end{equation}
Using that $\sqrt{x} \leq 1+x$ for every $x \geq 0$ and the integrability hypothesis in \eqref{ch2:property:a-regularity}, we deduce that there exist functions $\alpha \in L^1(\mathbb{R}^N)$ and $\beta \in L^\infty(\mathbb{R}^N)$ such that
\[
	|(\sqrt{a})_{ij}|(z) = \alpha(z) + \beta(z),
\]
for every $z \in \mathbb{R}^N$.
Considering the compact sets $K_{tx} \subset (0,\infty) \times \mathbb{R}^N_x$, $K_v \subset \mathbb{R}^N_v$, and $K_{v_*} \subset \mathbb{R}^N_{v_*}$, then
\begin{equation*}
\begin{split}
	\int_{K_{tx} \times K_v \times K_{v_*}} \left| \sqrt{a(v-v_*)} \right|_{ij} \sqrt{g_*} \sqrt{g}
	\leq &\Big(
			\| \alpha \|_{L^1}+ \| \beta \|_{L^\infty} |K_v|^{1/2} |K_{v_*}|^{1/2}
		\Big) \\
		& \times \|g\|^{1/2}_{L^1(K_{tx} \times K_v)} \|g\|^{1/2}_{L^1(K_{tx} \times K_{v_*})},
\end{split}
\end{equation*}
which implies the sequence $\sqrt{a(v-v_*)} \sqrt{f^n_*} \sqrt{f^n}$ is dominated by a locally integrable function.
By dominated convergence, the convergence \eqref{ch2:eq:entropy-dissipation-conv-1} holds in $L^1_{loc}((0,\infty) \times \mathbb{R}^{3N}_{xvv_*})$.

Next, the explicit expression \eqref{ch2:eq:sqrt-a} allows us to compute the divergence of this matrix, giving
\[
	\Div_z \left[ \sqrt{a(z)} \right]
	= - (N-1) \sqrt{\Gamma(|z|)} \frac{z}{|z|^2}.
\]
Using once again the inequality $\sqrt{x} \leq 1+x$, we deduce that
\[
	\left| \Div_z \big[ \sqrt{a(z)} \big] \right|
	\leq
	 \left| \Div_z  (a(z))  \right|
	+ (N-1)  \frac{1}{|z|}.
\]
The integrability condition \eqref{ch2:property:a-regularity} allows us to deduce that $\Div_z \sqrt{a(z)}$ is also in $L^1(\mathbb{R}^N) + L^\infty(\mathbb{R}^N)$, and, as before, we deduce that
\begin{equation}
	\label{ch2:eq:entropy-dissipation-conv-2}
	\begin{split}
	&\Div_v \left( \sqrt{a(v-v_*)} \right) \sqrt{f^n_* (1-\varepsilon_n f^n_*)} \frac{\arcsin \sqrt{\varepsilon_n f^n}}{\sqrt{\varepsilon_n f^n}} \sqrt{f^n}\\
	& \to
	\Div_v \left( \sqrt{a(v-v_*)} \right) \sqrt{f f_*}
	\end{split}
\end{equation}
in $L^1_{loc}((0,\infty) \times \mathbb{R}^{3N}_{xvv_*})$.

The convergences \eqref{ch2:eq:entropy-dissipation-conv-1} and \eqref{ch2:eq:entropy-dissipation-conv-2} imply that
\begin{equation}
	\label{ch2:eq:entropy-dissipation-conv-3}
	\begin{split}
	&\Div_v \left(
			\sqrt{a(v-v_*)}
			\sqrt{f^n_* (1-\varepsilon_n f^n_*)} \frac{2}{\sqrt{\varepsilon_n}}
			\arcsin\sqrt{\varepsilon_n f^n}
		\right)\\
		&  - \Div_v \left( \sqrt{a(v-v_*)} \right) \sqrt{f^n_* (1-\varepsilon_n f^n_*)} \frac{2}{\sqrt{\varepsilon_n}} \arcsin \sqrt{\varepsilon_n f^n}\\
	& \to 2 \Div_v \left(
		\sqrt{a(v-v_*)} \sqrt{f f_*}
	\right)
	- 2 \Div_v \left( \sqrt{a(v-v_*)} \right) \sqrt{f f_*}
	\end{split}
\end{equation}
in $\mathcal{D}'((0,\infty) \times \mathbb{R}^{3N}_{xvv_*})$.
Exchanging $v$ and $v_*$ and subtracting \eqref{ch2:eq:entropy-dissipation-conv-3} from it, we conclude that $D^n$ converges to
\begin{equation*}
\begin{split}
	D &\equiv 2 \Div_v \left(
			\sqrt{a(v-v_*)} \sqrt{f f_*}
		\right)
	- 2 \Div_v \left( \sqrt{a(v-v_*)} \right) \sqrt{f f_*}\\
	&\quad - 2 \Div_{v_*} \left(
			\sqrt{a(v-v_*)} \sqrt{f f_*}
		\right)
	+ 2 \Div_{v_*} \left( \sqrt{a(v-v_*)} \right) \sqrt{f f_*}
\end{split}
\end{equation*}
in $\mathcal{D}'((0,\infty) \times \mathbb{R}^{3N}_{xvv_*})$.

On the other hand, using \eqref{ch2:eq:entropy-inequality-proof-2} together with the bound $\varepsilon f_0^n \leq 1$, we obtain that
\[
	\frac{1}{2} \| D^n \|^2_{L^2} =
	\frac{1}{2} \int_0^t \intxv d^n \leq
	\intxv f^n_0 \log f^n_0 + \intxv f^n_0 - \intxv f^n \log f^n.
\]
Thus, $D^n$ is uniformly bounded in $L^2((0,\infty) \times \mathbb{R}^{3N}_{xvv_*})$, and we have $D^n \weakto D$ weakly in this space.
Then, from the lower semi-continuity of the $L^2$ norm with respect to the weak convergence, we have
\[
	\int_0^t \intxv d \equiv
	\| D \|^2_{L^2}
	\leq \liminf_{n \to \infty} \| D^n \|^2_{L^2}
	= \liminf_{n \to \infty} \int_0^t \intxv d^n
\]
which, together with \eqref{ch2:eq:entropy-inequality-proof-3}, leads to the result.

\section{Proof of Theorem \ref{ch2:thm:alternative-entropy}}
\label{ch2:sec:thm-2-proof}
We return to the convergence of the quadratic term
$-\beta_\delta''(f^n_n) \abar^{n,n}_{ij} \pdv{f^n_n}{v_i} \pdv{f^n_n}{v_j}$.
As before, we cannot prove this converges to something expressible as a function of the objects we have thus far, and as a consequence this term will be responsible for the appearance of a defect measure.
Indeed, let
\[
	\gamma_\delta(t) = - \frac{2\sqrt{2}}{\sqrt{\delta(1+\delta t)}}.
\]
We rewrite, in the sense of distributions,
\begin{align*}
	-\beta_\delta''(f^n_n) \abar^{n,n}_{ij} \pdv{f^n_n}{v_i} \pdv{f^n_n}{v_j}
	&= \abar^{n,n}_{ij} \pdv{\gamma_\delta(f^n_n)}{v_i} \pdv{\gamma_\delta(f^n_n)}{v_j}\\
	&= \langle \abar^{n,n} \nabla_v (\gamma_\delta(f^n_n)), \nabla_v (\gamma_\delta(f^n_n)) \rangle\\
	&= \left| \sqrt{\abar^{n,n}} \nabla_v (\gamma_\delta(f^n_n)) \right|^2.
\end{align*}
The vector field inside the squared norm can be rewritten as
\begin{equation}
	\label{ch2:eq:sqrt-abar-gamma-interp}
	\big( \sqrt{\abar^{n,n}} \big)_{ij} \pdv{\gamma_\delta(f^n_n)}{v_j}
	= \pdv{}{v_j} \left[
		\big( \sqrt{\abar^{n,n}} \big)_{ij} \gamma_\delta(f^n_n)
	\right]
	- \pdv{ \big( \sqrt{\abar^{n,n}} \big)_{ij} }{v_j} \gamma_\delta(f^n_n)
\end{equation}
and, as before, we want to apply some weak-strong convergence result here to at least deduce the convergence of this term in the sense of distributions.
However, it soon becomes clear that these are not directly applicable in this case.
Singularities in the derivatives of the matrix square root when eigenvalues are close to zero prevent us from concluding the convergence of the second term on the right-hand side of \eqref{ch2:eq:sqrt-abar-gamma-interp}.

What we can do then is approximate this matrix with one that is more regular and then reconstruct this term in the limit.
We follow an approach close to Villani's in \cite{villani_1996} and consider the approximation of the square root $S^k(M) = \chi^k(M) \sqrt{M}$, where $\chi^k$ is a $C^\infty(\mathbb{S}^n_+)$ function, defined on the space of positive semi-definite matrices $\mathbb{S}^n_+$, such that
\[
	\chi^k(M) =
	\begin{cases}
		1, \text{ if } \lambda_{max}(M) \leq k \text{ and } \lambda_{min}(M) \geq \frac{1}{k}\\
		0, \text{ if } \lambda_{max}(M) \geq 2k \text{ or } \lambda_{min}(M) \leq \frac{1}{2k},
	\end{cases}
\]
where $\lambda_{min}(M)$ and $\lambda_{max}(M)$ respectively denote the smallest and the largest eigenvalues of $M$.

In other words, we truncate $\sqrt{M}$ at high values, and at the singularity at zero.
We can show that $S^k$ is a $C^\infty$ function such that $S^k(M) \to M$ for every positive semi-definite matrix $M$.
Also, the sequence $(S^k(M))_k$ is increasing with respect to the matrix order, and we have $S^k(M) \leq \sqrt{M}$.

We approximate the quadratic term from below:
\begin{equation}
	\label{ch2:ineq:sk-uniform-in-k}
	\big(S^k(\abar^{n,n})^2\big)_{ij} \pdv{\gamma_\delta(f^n_n)}{v_i} \pdv{\gamma_\delta(f^n_n)}{v_j}
	\leq \abar^{n,n}_{ij} \pdv{\gamma_\delta(f^n_n)}{v_i} \pdv{\gamma_\delta(f^n_n)}{v_j}.
\end{equation}
Let us pass the approximation on the right-hand side to the limit.
We have
\[
	\big(S^k(\abar^{n,n})^2\big)_{ij} \pdv{\gamma_\delta(f^n_n)}{v_i} \pdv{\gamma_\delta(f^n_n)}{v_j}
	= \left| S^k(\abar^{n,n}) \nabla_v (\gamma_\delta(f^n_n)) \right|^2
\]
and for the term inside the absolute value, we have that
\begin{equation}
	\label{ch2:conv:sk-an-to-sk}
	\big( S^k(\abar^{n,n}) \big)_{ij} \pdv{\gamma_\delta(f^n_n)}{v_j}
	= \pdv{}{v_j} \left[
			\big( S^k(\abar^{n,n}) \big)_{ij} \gamma_\delta(f^n_n)
		\right]
	- \pdv{ \big( S^k(\abar^{n,n}) \big)_{ij} }{v_j} \gamma_\delta(f^n_n).
\end{equation}

Recall that from \eqref{ch2:ineq:coeff-conv-1}, the sequence $m_n$ was chosen such that $\abar^{n,m_n} \to \abar$ in $L^1_{loc}((0,\infty) \times \Rxv)$.
In this theorem, we have the extra convergence hypothesis \eqref{ch2:eq:nabla-v-a-nm-convergence} for the diffusion matrices of the approximation scheme.
Thus, for each $n$, we may choose the index $m_n$ sufficiently large, potentially increasing the values of the previously defined sequence $m_n$, such that we have not only the estimations \eqref{ch2:ineq:coeff-conv-1}, but also that
\[
	\| \nabla_v \abar^{n,m_n} - \nabla_v \abar^{n} \|_{L^1(B_n)} \leq \frac{1}{n},
\]
where $B_n = (0,n) \times B_x(0,n) \times B_v(0,n)$ is the $(t,x,v)$ ball with radius $n$.

In view of the extra regularity \eqref{ch2:eq:collision-kernel-extra-regularity} of the collision kernel, we have, from the convergence \eqref{ch2:conv:fn-1-fn-convergence}, that $\nabla_v \abar^n \to \nabla_v \abar$ in $L^1_{loc}((0,\infty) \times \Rxv)$, which therefore implies that
\begin{equation}
	\label{ch2:eq:nabla-v-abar-conv}
	\nabla_v \abar^{n,m_n}
	\to \nabla_v \abar
	\quad\text{in}\quad
	L^1_{loc}((0,\infty) \times \Rxv).
\end{equation}
As before, we relabel the sequences $\abar^{n,m_n}$ as $\abar^{n,n}$.

From the regularity of $S^k$, we have that $S^k(\abar^{n,n}) \to S^k(\abar)$ in $L^1_{loc}((0,\infty) \times \Rxv)$.
Since, as before, $\gamma_\delta(f^n_n) \to \gamma_\delta(f)$ in $L^1_{loc}((0,\infty) \times \Rxv)$, it follows that the first term on the right-hand side of \eqref{ch2:conv:sk-an-to-sk} converges to $\pdv{}{v_j} \left[ \big( S^k(\abar) \big)_{ij} \gamma_\delta(f) \right]$ in the sense of distributions.

Moreover, the regularity of $S^k$, together with the convergence \eqref{ch2:eq:nabla-v-abar-conv}, implies that $\pdv{ S^k(\abar^{n,n}) }{v_j} \to \pdv{ S^k(\abar) }{v_j}$ in $L^1_{loc}((0,\infty) \times \Rxv)$.
Therefore, the right-hand side of \eqref{ch2:conv:sk-an-to-sk} converges to
\[
	\pdv{}{v_j} \left[
		\big( S^k(\abar) \big)_{ij} \gamma_\delta(f)
	\right]
	- \pdv{ \big( S^k(\abar) \big)_{ij} }{v_j} \gamma_\delta(f)
\]
in $\mathcal{D}'((0,\infty) \times \Rxv)$.
We denote this limit as $\big( S^k(\abar) \big)_{ij} \pdv{\gamma_\delta(f)}{v_j}$.

Let $S^k(\abar) \nabla_v \gamma_\delta(f)$ denote the vector field whose $i$-th component is $\big( S^k(\abar) \big)_{ij} \pdv{\gamma_\delta(f)}{v_j}$.
From the bound \eqref{ch2:ineq:sk-uniform-in-k} and   Lemma \ref{ch2:lem-aij-v-bounded}, we deduce that the sequence $\big(S^k(\abar^{n,n}) \nabla_v (\gamma_\delta(f^n_n))\big)_n$ is uniformly bounded in $L^2_{loc}((0,\infty) \times \Rxv)$.
Therefore, we have the weak convergence
\begin{equation}
	\label{ch2:conv:sk-an-l2-weak-conv}
	S^k(\abar^{n,n}) \nabla_v (\gamma_\delta(f^n_n))
	\weakto
	S^k(\abar) \nabla_v (\gamma_\delta(f))
	\quad\text{in }
	L^2_{loc}((0,\infty) \times \Rxv).
\end{equation}
Thus, the distribution $S^k(\abar) \nabla_v \gamma_\delta(f)$ can be identified with an $L^2_{loc}$ function.

By the lower semicontinuity of the $L^2$ norm with respect to the weak convergence in \eqref{ch2:conv:sk-an-l2-weak-conv}, for every compact set $K \subset (0,\infty) \times \Rxv$, we have
\begin{equation}
	\label{ch2:eq:sk-nabla-gamma-bound-lim-n}
	\int_K
	| S^k(\abar) \nabla_v (\gamma_\delta(f)) |^2
	\leq \liminf_{n \to \infty} \int_K
	| S^k(\abar^{n,n}) \nabla_v (\gamma_\delta(f^n_n)) |^2.
\end{equation}
Combining this inequality with \eqref{ch2:ineq:sk-uniform-in-k} and Lemma \ref{ch2:lem-aij-v-bounded}, we deduce that the sequence $(S^k(\abar) \nabla_v (\gamma_\delta(f)))_k$ is uniformly bounded in $L^2_{loc}((0,\infty) \times \Rxv)$.
Therefore, passing to a subsequence in $k$, there exists a weak limit, which we denote by $\sqrt{\abar} \nabla_v (\gamma_\delta(f))$.
That is,
\[
	S^k(\abar) \nabla_v (\gamma_\delta(f))
	\weakto
	\sqrt{\abar} \nabla_v (\gamma_\delta(f))
	\quad \text{in } L^2_{loc}((0,\infty) \times \Rxv),
\]
Applying lower semicontinuity again, for every compact set $K \subset (0,\infty) \times \Rxv$, we have
\begin{equation}
\label{ch2:eq:sk-nabla-gamma-bound-lim-k}
	\int_K \big|
		\sqrt{\abar} \nabla_v (\gamma_\delta(f))
	\big|^2
	\leq
	\liminf_{k \to \infty} \int_K
	\big|
		S^k(\abar) \nabla_v (\gamma_\delta(f))
	\big|^2.
\end{equation}
We formally identify the term $-\beta_\delta''(f) \bar{a}_{ij} \partial_{v_i} f \partial_{v_j} f$ with the square of this limit function $|\sqrt{\bar{a}} \nabla_v \gamma_\delta(f)|^2$.

Next, we establish the existence of the defect measure.
We aim to show that there exists a nonnegative Radon measure $\mu \in \mathcal{M}_{loc}((0,\infty) \times \Rxv)$ such that
\[
	\abar^{n,n}_{ij} \pdv{\gamma_\delta(f^n_n)}{v_i} \pdv{\gamma_\delta(f^n_n)}{v_j}
	\weakto
	-\beta_\delta''(f) \abar_{ij} \pdv{f}{v_i} \pdv{f}{v_j} + \mu
	\quad \text{as } n \to \infty.
\]
Using Lemma \ref{ch2:lem-aij-v-bounded} and the Banach-Alaoglu theorem, we deduce the existence of a Radon measure $\nu$ such that
\[
	\abar^{n,n}_{ij} \pdv{\gamma_\delta(f^n_n)}{v_i} \pdv{\gamma_\delta(f^n_n)}{v_j}
	\weakto
	\nu.
\]
We define the measure $\mu$ by
\[
	\mu = \nu - \left(
		- \beta_\delta''(f) \abar_{ij} \pdv{f}{v_i} \pdv{f}{v_j}
	\right).
\]
Since the subtracted term is the square of an $L^2_{loc}$ function, it belongs to $L^1_{loc}$, ensuring that $\mu \in \mathcal{M}_{loc}((0,\infty) \times \Rxv)$.
It remains to show that $\mu$ is nonnegative.

Indeed, from \eqref{ch2:ineq:sk-uniform-in-k}, we have that for every compact set $K \subset (0,\infty) \times \Rxv$, one has
\begin{equation*}
	\int_K \abar^{n,n}_{ij} \pdv{\gamma_\delta(f^n_n)}{v_i} \pdv{\gamma_\delta(f^n_n)}{v_j}
	\geq \int_K | S^k(\abar^{n,n}) \nabla_v (\gamma_\delta(f^n_n)) |^2.
\end{equation*}
We take the $\liminf_{n \to \infty}$ on both sides.
Using \eqref{ch2:eq:sk-nabla-gamma-bound-lim-n} and then passing both sides to the $\liminf_{k \to \infty}$ gives
\begin{equation*}
\begin{split}
	\liminf_{n\to\infty} \int_K \abar^{n,n}_{ij} \pdv{\gamma_\delta(f^n_n)}{v_i} \pdv{\gamma_\delta(f^n_n)}{v_j}
	&\geq \liminf_{k\to\infty} \int_K
		| S^k(\abar) \nabla_v (\gamma_\delta(f)) |^2 \\
	&\geq \int_K \big|
			\sqrt{\abar} \nabla_v (\gamma_\delta(f))
		\big|^2,
\end{split}
\end{equation*}
where in the last inequality we used \eqref{ch2:eq:sk-nabla-gamma-bound-lim-k}.
Therefore, we have that
\[
	\liminf_{n\to\infty}
	\int_K
	\left[
		\abar^{n,n}_{ij} \pdv{\gamma_\delta(f^n_n)}{v_i} \pdv{\gamma_\delta(f^n_n)}{v_j}
		- \left(
			- \beta_\delta''(f) \abar_{ij} \pdv{f}{v_i} \pdv{f}{v_j}
		\right)
	\right]
	\geq 0,
\]
for every compact set $K \subset (0,\infty) \times \Rxv$, which implies that the measure $\mu$ is nonnegative.

\bibliographystyle{acm}
\bibliography{references}

\end{document}